\documentclass[11pt]{amsart}

\usepackage{latexsym}
\usepackage{amssymb}
\usepackage{amsmath}
\usepackage{fancybox,color}
\usepackage{amssymb}
\usepackage{amsfonts}
\usepackage{latexsym}
\usepackage{amsmath,systeme}
\usepackage[leqno]{mathtools}

\usepackage{amsmath}
\usepackage{amsthm}
\usepackage{color}
\usepackage{amssymb}
\usepackage{graphicx}
\usepackage{mathrsfs}
\usepackage{multicol}
\usepackage{wrapfig}
\usepackage{times}
\usepackage{enumerate}
\newtheorem{definition}{Definition}
\newtheorem{remark}{Remark}
\newtheorem{Lemma}{Lemma}
\newtheorem{theorem}{Theorem}

\newtheorem{corol}{Corollary}

\usepackage{mathtools}
\mathtoolsset{showonlyrefs}

\def\vint_#1{\mathchoice%
          {\mathop{\kern 0.2em\vrule width 0.6em height 0.69678ex depth -0.58065ex
                  \kern -0.8em \intop}\nolimits_{\kern -0.4em#1}}%
          {\mathop{\kern 0.1em\vrule width 0.5em height 0.69678ex depth -0.60387ex
                  \kern -0.6em \intop}\nolimits_{#1}}%
          {\mathop{\kern 0.1em\vrule width 0.5em height 0.69678ex
              depth -0.60387ex
                  \kern -0.6em \intop}\nolimits_{#1}}%
          {\mathop{\kern 0.1em\vrule width 0.5em height 0.69678ex depth -0.60387ex
                  \kern -0.6em \intop}\nolimits_{#1}}}
\def\vintslides_#1{\mathchoice%
          {\mathop{\kern 0.1em\vrule width 0.5em height 0.697ex depth -0.581ex
                  \kern -0.6em \intop}\nolimits_{\kern -0.4em#1}}%
          {\mathop{\kern 0.1em\vrule width 0.3em height 0.697ex depth -0.604ex
                  \kern -0.4em \intop}\nolimits_{#1}}%
          {\mathop{\kern 0.1em\vrule width 0.3em height 0.697ex depth -0.604ex
                  \kern -0.4em \intop}\nolimits_{#1}}%
          {\mathop{\kern 0.1em\vrule width 0.3em height 0.697ex depth -0.604ex
                  \kern -0.4em \intop}\nolimits_{#1}}}

\newcommand{\kint}{\vint}

\newcommand{\aveint}[2]{\mathchoice%
          {\mathop{\kern 0.2em\vrule width 0.6em height 0.69678ex depth -0.58065ex
                  \kern -0.8em \intop}\nolimits_{\kern -0.45em#1}^{#2}}%
          {\mathop{\kern 0.1em\vrule width 0.5em height 0.69678ex depth -0.60387ex
                  \kern -0.6em \intop}\nolimits_{#1}^{#2}}%
          {\mathop{\kern 0.1em\vrule width 0.5em height 0.69678ex depth -0.60387ex
                  \kern -0.6em \intop}\nolimits_{#1}^{#2}}%
          {\mathop{\kern 0.1em\vrule width 0.5em height 0.69678ex depth -0.60387ex
                  \kern -0.6em \intop}\nolimits_{#1}^{#2}}}
\newcommand{\ol}{\overline}
\newcommand{\ul}{\underline}
\newcommand{\Om}{\Omega}
\newcommand{\R}{\mathbb{R}}

\newcommand{\E}{\mathbb{E}}

\newcommand{\A}{\mathscr{A}}

\newcommand{\eps}{\varepsilon}

\newcommand{\half}{\frac{1}{2}}

\begin{document}
\large
\title[Games for the two membranes problem]{Games for the two membranes problem}
\author[Alfredo Miranda and Julio D. Rossi]{Alfredo Miranda and Julio D. Rossi}

\address{Alfredo Miranda and Julio D. Rossi
\hfill\break\indent
Departamento  de Matem{\'a}tica, FCEyN,
Universidad de Buenos Aires,
\hfill\break\indent Pabellon I, Ciudad Universitaria (1428),
Buenos Aires, Argentina.}

\email{{\tt amiranda@dm.uba.ar , jrossi@dm.uba.ar}}

\begin{abstract}
	In this paper we find viscosity solutions to the two membranes problem (that is a system with two obstacle-type equations) with two different $p-$Laplacian operators taking limits of value functions of a sequence of games. We analyze two-player zero-sum games that are played in two boards with different rules in each board. At each turn both players (one inside each board) have the choice of playing without changing board or to change to the other board (and then play one round of the other game). We show that the value functions corresponding to this kind of games converge uniformly to a viscosity solution of the two membranes problem. If in addition the possibility of having the choice to change boards depends on a coin toss we show that we also have convergence of the value functions to the two membranes problem that is supplemented with an extra condition inside the coincidence set.
\end{abstract}

\date{}

\maketitle

\section{Introduction} 

The deep connection between Partial Differential Equations and Probability
is a well known and widely studied subject. For linear operators, like the Laplacian,
this relation turns out to rely on the validity of mean value formulas
for the solutions in the PDE side and martingale identities in the
probability side. In fact, there is a standard connection between 
 the Laplacian and the Brownian motion or with
the limit of random walks as the step size goes to zero, see, for example,
\cite{Doob,Doob2,Doob3,Hunt,Kac,Kaku,FFF,Williams}.  
Recently, starting with \cite{PSSW} some of these connections were extended
to cover nonlinear equations.
For a probabilistic approximation of the infinity Laplacian there is a game (called Tug-of-War game in
the literature) that was introduced in \cite{PSSW} whose value functions
approximate solutions to the PDE as a parameter that controls the size of the steps in the game goes to zero. 
In \cite{PS}, see also \cite{MPRa} and  \cite{MPRb}, the authors introduce a modification of the game (called Tug-of-War with noise) that is related to the normalized $p-$Laplacian. 
Approximation of solutions to 
linear and nonlinear PDEs using game theory is by now a classical subject.
The previously mentioned results were extended to cover very different equations (like Pucci operators, the Monge-Ampere equation, the obstacle problem, etc), we refer to the recent books 
\cite{BRLibro} and \cite{Lewicka} and references therein.  
However, much less is known concerning 
the relation between systems of PDEs and games, see \cite{nosotros,nosotros2,Mitake} for recent
references. One of the systems that attracted the attention 
of the PDE community is the two membrane problem. This problem models
the behaviour of two elastic membranes that are clamped at the boundary of a prescribed domain (they are assumed to be ordered, one membrane is 
assumed to be above the other)
and they are subject to different external forces
(the membrane that is on top is pushed down and the one that is below is pushed up). 
The main assumption here is that the two membranes do not penetrate each other
(they are assumed to be ordered in the whole domain). This situation can be
modeled by a two obstacle problem; 
 the lower membrane acts as an obstacle from below for the free elastic equation
 that describes the location of the upper membrane, while, conversely, the upper 
membrane is an obstacle from above for the equation for the lower membrane. 
When the equations that obey the two membranes have a variational structure
this problem can be tackled using calculus of variations (one aims to minimize
the sum of the two energies subject to the constraint that the functions that describe the position of the membranes are always ordered inside the domain, one is bigger or equal than the other), see \cite{VC}. However, when the involved equations are not variational the analysis
relies on monotonicity arguments (using the maximum principle). 
Once existence of a solution (in an appropriate sense) is obtained a lot of interesting 
questions arise, like uniqueness, regularity of the involved functions, a description of the contact set, the regularity of the contact set, etc. See \cite{Caffa1,Caffa2,S}, the 
dissertation \cite{Vivas}
and references therein.

Our main goal in this paper is to analyze games whose value functions approximate solutions to the two membranes problem with two different normalized $p-$Laplacians (these are fully nonlinear non-variational equations, see below). 

\subsection{The normalized $p-$Laplacian and game theory}
To begin, let us introduce the normalized $p-$Laplacian and give the relation between this
operator and the game called Tug-of-War with noise in the literature,
we refer to \cite{MPRb} and the recent books \cite{BRLibro} and \cite{Lewicka} for 
details. This kind of game was extensively studied, some
related references are \cite{TPSS,AS,BlancPR,BR,ChGAR,ML,QS,LPS,MPR,MPRa,MPRb,MPRparab,MRS}.

Consider the classical $p-$Laplacian operator $\Delta_p u=\mbox{div}(|\nabla u |^{p-2}\nabla u)$ with $2\leq p<\infty$. Expanding the divergence we can write (formally) this operator as a combination of the Laplacian operator $\Delta u=\sum_{n=1}^{N} u_{x_n x_n}$ and the $1$-homogeneous infinity Laplacian $\Delta_{\infty}^1 u=\langle D^2 u \frac{\nabla u}{|\nabla u|} , \frac{\nabla u}{|\nabla u|} \rangle =|\nabla u|^{-2}\sum_{1\leq i,j\leq N}u_{x_i}u_{x_i x_j}u_{x_j}$ as follows
\begin{equation}
\label{PLap}
\Delta_p u =|\nabla u |^{p-2}\big((p-2)\Delta^1_{\infty}u+\Delta u\big).
\end{equation}
Now we want to recall the mean value formula associated to this operator
obtained in \cite{MPR} (see also \cite{Arroyo} and \cite{ML}). Given $0<\alpha< 1$, let us consider $u:\Om\rightarrow\R$ such that 
\begin{equation}
\label{DPPPLap}
u(x)=\alpha\big[\half\sup_{y \in B_{\eps}(x)}u(y)+\half\inf_{y \in B_{\eps}(x)}u(y)\big]+(1-\alpha)\kint_{B_{\eps}(x)}u(y)dy + o(\eps^2),
\end{equation}
as $\eps\to 0$.
It turns out that $u$ verifies this asymptotic mean value formula if and only if $u$ is a viscosity solution to $\Delta_p u=0$, see \cite{MPR}. For general references on mean value formulas for solutions to nonlinear PDEs we refer to \cite{Arroyo,BChMR,I,KMP,ML}.
In fact, if we assume that $u$ is smooth, using a simple Taylor expansion we have
\begin{equation} \label{asymp-1}
\kint_{B_{\eps}(x)}u(y)dy-u(x)=\frac{\eps^2}{2(N+2)}\Delta u(x)+o(\eps^2),
\end{equation}
and if $|\nabla u (x)|\neq 0$, using again a simple Taylor expansion we obtain
\begin{equation} \label{asymp-2}
\begin{array}{ll}
\displaystyle \big[\half\sup_{y \in B_{\eps}(x)}u(y)+\half\inf_{y \in B_{\eps}(x)}u(y)\big]-u(x)\sim\big[\half u(x+\eps\frac{\nabla u(x)}{|\nabla u(x)|})+\half u(x-\eps\frac{\nabla u(x)}{|\nabla u(x)|})\big]\\[8pt]
\displaystyle =\frac{\eps^2}{2}\Delta^1_{\infty}u(x)+o(\eps^2).
\end{array}
\end{equation}
Then, if we come back to \eqref{DPPPLap}, divide by $\eps^2$, and take $\eps\to 0$, we get
\begin{equation}
\displaystyle 0=\frac{\alpha}{2}\Delta^1_{\infty}u(x)+\frac{(1-\alpha)}{2(N+2)}\Delta u(x).
\end{equation}
Thus, from \eqref{PLap}, we get that the function $u$ is a solution to the equation 
\begin{equation}
-\Delta_p u(x)=0
\end{equation}
for $p> 2$ such that 
\begin{equation}
\label{palpha}
\frac{\alpha}{1-\alpha}=\frac{p-2}{N+2}.
\end{equation}
These computations can be made rigorous using viscosity theory, we refer to \cite{MPR}.

Now, let us introduce the normalized $p-$Laplacian (also called the game $p-$Laplacian).
\begin{definition}
	\label{def.plaplaciano}
	Given $p>2$, let the normalized $p-$Laplacian be defined as
	\begin{equation}
	\label{defpLap}
	\Delta_p^{1} u(x)=\frac{\alpha}{2}\Delta_{\infty}^1 u(x)+\frac{1-\alpha}{2(N+2)}\Delta u(x),
	\end{equation}
	with 
\begin{equation}
\frac{\alpha}{1-\alpha}=\frac{p-2}{N+2}.
\end{equation} 
\end{definition}

Remark that this is a nonlinear elliptic $1-$homogeneous operator
that is a linear combination between the classical Laplacian and the $\infty$-Laplacian. 


There is a game theoretical approximation 
to these operators. The connection between the Laplacian and the Brownian motion or with
the limit of random walks as the step size goes to zero is well known, see 
\cite{Kac,Kaku,FFF}.  
For a probabilistic approximation of the infinity Laplacian there is a game (called Tug-of-War game in
the literature) that was introduced in \cite{PSSW}. 
In \cite{PS}, see also \cite{MPRa} and  \cite{MPRb}, the authors introduce a two-player zero-sum game called Tug-of-War with noise that is related to the normalized $p-$Laplacian.
This is a two-player zero-sum game (two players, Player I and Player II, play one against the other and the total earnings of one player are exactly the
total losses of the other). The rules of the game are as follows: In a bounded smooth domain $\Om$ (what we need here is that $\partial\Omega$ satisfies an exterior sphere condition) given an initial position $x\in\Om$, with probability $\alpha$ Player I and Player II play Tug-of-War (the players toss a fair coin and the winner chooses the next position of the token in $B_{\eps}(x)$), and with probability  $(1-\alpha)$ they move at random (the next position of the token is chosen at random in $B_{\eps}(x)$). They continue playing with these rules until
the game position leaves the domain $\Om$. At this stopping time Player II pays Player I
the amount determined by a pay-off function defined outside $\Om$. The value of the game
(defined as the best value that both players may expect to obtain) verifies a mean value formula, called the Dynamic Programming Principle (DPP), that in this case is given by
\begin{equation}
u^{\eps}(x)=\alpha\big[\half\sup_{y \in B_{\eps}(x)}u^{\eps}(y)+\half\inf_{y \in B_{\eps}(x)}u^{\eps}(y)\big]+(1-\alpha)\kint_{B_{\eps}(x)}u^{\eps}(y)dy.
\end{equation}
This is exactly the same formula as \eqref{DPPPLap} but without the error term $o(\eps^2)$.
Notice that the value function of this game depends on $\eps$, the
parameter that controls the size of the possible movements.
Remark that this equation can be written as 
\begin{equation}
0=\alpha\big[\half\sup_{y \in B_{\eps}(x)}u^{\eps}(y)+\half\inf_{y \in B_{\eps}(x)}u^{\eps}(y) 
- u^{\eps}(x) \big]+(1-\alpha)\kint_{B_{\eps}(x)}(u^{\eps}(y)-u^{\eps}(x))dy.
\end{equation}
Using as a main tool the asymptotic formulas \eqref{asymp-1} and \eqref{asymp-2}, in \cite{BlancPR} and \cite{MPRb} the authors show that there is a uniform limit
as $\eps \to 0$,
$$
u^\eps \rightrightarrows u
$$
and that this limit $u$ is the unique solution (in a viscosity sense) to the Dirichlet problem
\begin{equation}
\left\lbrace
\begin{array}{ll}
\displaystyle -\Delta^1_p u(x) =0, & x\in \Om, \\[10pt]
\displaystyle u(x)=F(x), & x\in\partial\Om.
\end{array}
\right.
\end{equation}
When one wants to deal with a non-homogeneous equation like $-\Delta^1_p u(x) = h(x)$
one can add a running payoff to the game, that is, at every play Player I pays to Player II
the amount $\eps^2 h(x)$.

\subsection{The two membranes problem} As we already mentioned, the two membranes problem describes the equilibrium position of two elastic
membranes in contact with each other that are not allowed to cross. Hence, one
of the membranes acts as an obstacle (from above or below) for the other. 
Given two differential operators $F(x,u,\nabla u, D^2u)$ and $G
(x,v,\nabla v, D^2v)$ the mathematical formulation the two membranes problem is 
the following:
\begin{equation}
	\label{ED12FG}
	\left\lbrace 
	\begin{array}{ll}
		\displaystyle \min\Big\{ F(x,u(x) ,\nabla u(x), D^2u(x)),(u-v)(x)\Big\}=0, \quad & x\in\Om, \\[10pt]
	\displaystyle \max\Big\{ G(x,v(x) ,\nabla v(x), D^2v(x)),(v-u)(x)\Big\}=0, \quad & x\in\Om, \\[10pt]
	u(x)=f(x), \quad & x\in\partial\Om, \\[10pt]
	v(x)=g(x), \quad & x\in\partial\Om.
	\end{array}
	\right.
	\end{equation}
	
	In general there is no uniqueness for the two membranes problem.
For example, take $u$ the solution to the first operator $F(u)=0$ with $u|_{\partial \Omega} = f$ and $v$ the solution to the obstacle problem from above for $G(v)$ and boundary datum $g$. This pair $(u,v)$ is a solution to the two membranes problem
($v$ is a solution to the obstacle problem with $u$ as upper obstacle and $u$ is a solution 
to the obstacle problem with $v$ as lower obstacle (in fact $u$ is a solution 
in the whole domain and is above $v$)). Analogously, one can consider 
$\tilde{v}$ as the solution to $G(\tilde{v})=0$ with $v|_{\partial \Omega} = g$ and $\tilde{u}$ the solution to the obstacle problem from above for $F(\tilde{u})$ and boundary datum $f$, to obtain
a pair $(\tilde{u}, \tilde{v})$ that is a solution to the two membranes problem. 
In general, it holds that $(u,v) \neq (\tilde{u}, \tilde{v})$. 
	
The two membranes problem for the Laplacian with a right hand side, that is, for $F(D^2u)=-\Delta u +h_1$ and $F(D^2v)=-\Delta v-h_2$, was first considered in \cite{VC} 
using variational arguments. Latter, in \cite{Caffa1} the authors solve the two membranes problem for two different fractional Laplacians of different order (two 
linear non-local operators defined by two different kernels). 
Notice that in this case the problem is still variational. In these cases an extra condition appears, namely, the sum of the two operators vanishes,
\begin{equation}
\label{extra-caffa}
G(u) + F (v) =0,
\end{equation} 
inside $\Omega$. Moreover, this extra condition together with the variational structure is used to prove a $C^{1,\gamma}$ regularity result for the solution.

The two membranes problem for a nonlinear operator was
studied in \cite{Caffa1, Caffa2,S}.
In particular, in \cite{Caffa2} the authors consider a version of
the two membranes problem for two different fully nonlinear operators, $F(D^2 u)$ and
$G(D^2 u)$. Assuming that $F$ is convex and that
\begin{equation} 
\label{cond-caffaVivas}
G(X) = -F(-X),
\end{equation} 
they prove that solutions are $C^{1,1}$ smooth. 

We also mention that a more general version of the two membranes problem
involving more than two membranes was considered by several authors (see for example  \cite{ARS,CChV,ChV}).

\subsection{Description of the main results}
In this paper we use the previously described Tug-of-War with noise game to obtain games whose value functions approximate solutions (in a viscosity sense) to a system with two obstacle type equations (a two membrane problem).

\subsubsection{First game} \label{subsect-1st}
Let us describe the first game that we are going to study in more detail. Again, it is a two-player zero-sum game. The game is played in two boards,
that we call board 1 and board 2, that are two copies of a ﬁxed smooth bounded domain $\Om\subset \R^N$. We fix two final payoff functions $f,g:\R^N\backslash \Om \rightarrow \R$  two uniformly Lipschitz functions with $f\geq g$, and two running payoff functions $h_1,h_2:\Om\rightarrow \R$ (we also assume that they are uniformly Lipschitz functions), corresponding to the first and second board respectively. Take a positive parameter $\eps$. Let us use two games with different rules for the first and second board respectively associated to two different $p-$Laplacian operators. To this end, let us fix two numbers $0<\alpha_i<1$ for $i=1,2$. Playing in the first board the rules are the following: with $\alpha_1$ probability we play with Tug-of-War rules, this means a fair coin is tossed and the player who the coin toss chooses the next position inside the ball $B_{\eps}(x)$, and with $(1-\alpha_1)$ probability we play with a random walk rule, the next position is chosen at random in $B_{\eps}(x)$ with uniform probability. Playing in the first board we add a running payoff of amount $-\eps^2 h_1(x_0)$ (Player I gets $-\eps^2 h_1(x_0)$ and Player II $\eps^2 h_1(x_0)$). We call this game the $J_1$ game. Analogously, in the second board we use 
$\alpha_2$ to encode the probability that we play Tug-of-War and $(1-\alpha_2)$ for the probability to move at random, this time
with a running payoff of amount $\eps^2 h_2 (x_0)$. We call this game $J_2$.

To the rules that we described in the two boards $J_1$ and $J_2$
 we add the following ways of changing boards: in the first board, Player I decides to play
with $J_1$ rules (and the game position remains at the first board) or to change boards and the new position of the token is chosen playing the $J_2$ game rule in the second board. In the second board the rule is just the opposite, in this case Player II decides to play with $J_2$ game rules (and remains at the second board) or to change boards and play in the first board with the $J_1$ game rules. 

The game starts with a token at an initial position $x_0\in\Om$ in one of the two boards. 
After the first play the game continues with the same rules
(each player decides to change or continue in one board plus the rules for the two diffenet Tug-of-War with noise game at each board) until the tokes leaves 
the domain $\Omega$ (at this time the game ends). This gives a random sequence of points (positions of the token)
and a stopping time $\tau$ (the first time that the position of the token is outside $\Omega$ in any of the
two boards). The sequence of positions will be denoted by
$$
\Big\{(x_0,j_0),(x_1,j_1),\dots (x_\tau, j_\tau) \Big\},
$$
here $x_k \in \Omega$ (and $x_\tau \not\in \Omega$) and the second variable, $j_k\in\{1,2\}$, is
just an index that indicates in which board we are playing, $j_k=1$ if the position of the token is in the first board and $j_k=2$ if we are in
the second board. As we mentioned, the game ends when the token leaves $\Omega$ at some 
point $(x_{\tau},j_\tau)$. In this case the final payoff (the amount that Player I gets and Player II pays)
is given by $f(x_\tau)$ if $j_\tau =1$ (the token leaves the domain in the first board) and 
$g(x_\tau)$ if $j_\tau =2$ (the token leaves in the second board). Hence, taking 
into account the running payoff and the final payoff, the total payoff of 
a particular occurrence of 
the game is given by
$$
\mbox{total  payoff} : = f(x_{\tau})\chi_{\{1\}}(j_{\tau})+g(x_{\tau})\chi_{\{2\}}(j_{\tau})-\eps^2\sum_{k=0}^{\tau -1}\Big(h_1 (x_k)\chi_{\{1\}}(j_{k+1})-h_2 (x_k)\chi_{\{j=2\}}(j_{k+1})\Big).
$$
Notice that the total payoff is the sum of the final payoff (given by $f(x_{\tau})$ or by $g(x_{\tau})$
according to the board at which the position leaves the domain) and the running payoff
that is given by $-\eps^2 h_1(x_k)$ and $\eps^2 h_2(x_k)$ corresponding to the board in which we play at step $k+1$. 

Now, the players fix two strategies, $S_{I}$ for Player I and $S_{II}$ for Player II. That is, both players decide to play or to change boards in the respective board, and in each board they select the point to go 
provided the coin toss of the Tug-of-War game is favorable. Notice that
the decision on the board where the game takes place is made by the players
at each turn (according to the board at which the position is one of the players makes the choice). Hence, when the strategies of both players are fixed, the board in which the game occurs at each turn 
is given (and it is not random).
Then, once we fix the strategies 
$S_{I}$ and $S_{II}$, everything depends only on the underlying probability: the coin toss that decides when to play Tug-of-War and when to move at random (remark that this probability is given by $\alpha_1$ or $\alpha_2$ and it is different in the two bards) 
and the coin toss (with probability 1/2--1/2) that decides who choses the 
next position of the game if the Tug-of-War game is played. With respect to this underlying probability, with fixed strategies $S_{I}$ and $S_{II}$, we can compute the expected
final payoff starting at $(x,j)$ (recall that $j=1,2$ indicates the board at which is the position of the game),
$$ 
\E_{S_{I},S_{II}}^{(x,j)}[\mbox{total payoff}].
$$

The game is said to have a value if
\begin{equation} \label{value}
w^{\eps}(x,j)=\sup_{S_{I}}\inf_{S_{II}}\E_{S_{I},S_{II}}^{(x,j)}[\mbox{total payoff}] 
= \inf_{S_{II}}\sup_{S_{I}}\E_{S_{I},S_{II}}^{(x,j)}[\mbox{total payoff}].
\end{equation}
 Notice that this value $w^\eps$ is the best possible expected outcome that Player I and
 Player II may expect to obtain playing their best. Here we will prove that this game has a value.
 The value of the game, $w^\eps$, is composed in fact by two functions, the first one defined
 in the first board,
 $$
 u^{\eps}(x) := w^\eps (x, 1)
 $$
 that is the expected outcome of the game if the initial position is at the first board
 (and the players play their best)
 and 
 $$
 v^{\eps}(x) := w^\eps (x, 2)
 $$ 
that is the expected outcome of the game when the initial position is in the second board.
It turns out that these two functions $u^\eps$, $v^\eps$ satisfy a system of equations
that is called the Dynamic Programming Principle (DPP) in the literature.
In our case, the corresponding (DPP) for the game is given by
\begin{equation}
\label{DPP}
\left\lbrace
\begin{array}{ll}
\displaystyle u^{\eps}(x)=\max\Big\{ J_1(u^{\eps})(x), J_2(v^{\eps})(x)\Big\} ,
\qquad &  x \in \Omega,  \\[10pt]
\displaystyle v^{\eps}(x)=\min\Big\{ J_1(u^{\eps})(x), J_2(v^{\eps})(x)\Big\} ,
&  x \in \Omega,  \\[10pt]
u^{\eps}(x) = f(x), \qquad & x \in \R^{N} \backslash \Omega,  \\[10pt]
v^{\eps}(x) = g(x), \qquad &  x \in \R^{N} \backslash \Omega.
\end{array}
\right.
\end{equation}
where 
\begin{equation}
J_1(w)(x)=\alpha_1\Big[\half \sup_{y \in B_{\eps}(x)}w(y) + \half \inf_{y \in B_{\eps}(x)}w(y)\Big]+(1-\alpha_1)\kint_{B_{\eps}(x)}w(y)dy-\eps^2h_1(x)
\end{equation}
and
\begin{equation}
J_2(w)(x)=\alpha_2\Big[\half \sup_{y \in B_{\eps}(x)}w(y) + \half \inf_{y \in B_{\eps}(x)}w(y)\Big]+(1-\alpha_2)\kint_{B_{\eps}(x)}w(y)dy+\eps^2h_2(x).
\end{equation}

\begin{remark} {\rm
		From the (DPP) and the condition $f\geq g$ it is clear that the value functions of the game are ordered. We have $$u^{\eps}(y)\geq v^{\eps}(y)$$ for all $y\in\R^N$.}
\end{remark} 

\begin{remark}{\rm
	 Observe that the (DPP) reflects the rules for the game described above. That is, the $J_1$ rule says that with $\alpha_1$ probability we play with Tug-of-war game and with $(1-\alpha_1)$ probability we play the random walk game with a running payoff that involves $h_1$. Analogously, in the $J_2$ game the probability is given by $\alpha_2$ and the running payoff involves $h_2$. Also notice that the $\max$ and $\min$ that are arise in the (DPP) corresponds to the choices of the players to change board (or not). In the first board 
	 the first player (who aims to maximize the expected outcome) is the one who decides 
	 while in the second board the second player (that wants to minimize) decides. }
\end{remark}

Our first result says that the value functions of the game 
converge uniformly as $\eps \to 0$ to a pair of continuous functions $(u,v)$ 
that is a viscosity solution to a system of partial differential equations in which
two equations of obstacle type appear. 

\begin{theorem} 
	\label{teo.limite}
	There exists a sequence $\eps_j \to 0$ such that $(u^{\varepsilon_j}, v^{\varepsilon_j})$ converges to a pair of continuous functions $(u,v)$,
	$$
	u^{\varepsilon_j}\rightrightarrows u, \qquad v^{\varepsilon_j}\rightrightarrows v
	$$
	uniformly in $\overline{\Omega}$. 
	The limit pair is a viscosity solution to the two membrane system
	with two different $p-$Laplacians, that is, 
	\begin{equation}
	\label{ED1}
	\left\lbrace
	\begin{array}{ll}
	\displaystyle u (x) \geq v(x), \qquad & \ x\in \Omega, \\[10pt]
	\displaystyle -\Delta_{p}^{1}u(x)+ h_1(x)\geq 0, \quad  \quad -\Delta_q^1 v(x)- h_2(x)\leq 0, \qquad & \ x\in \Om,\\[10pt] 
	\displaystyle -\Delta_p^1 u(x)+ h_1(x)=0, \quad  \quad -\Delta_q^1 v(x)- h_2(x)=0,\qquad & \ x\in \{u>v\}\cap\Om,\\[10pt] 
	u(x) = f(x), \qquad & \ x \in \partial \Omega,  \\[10pt]
	v(x) = g(x), \qquad & \ x \in \partial \Omega.
	\end{array}
	\right.
	\end{equation}
	Here $p$ and $q$ are given by
	\begin{equation} \label{p,q}
	\frac{\alpha_1}{1-\alpha_1}=\frac{p-2}{N+2} \quad \mbox{and} \quad \frac{\alpha_2}{1-\alpha_2}=\frac{q-2}{N+2}.
	\end{equation}
\end{theorem}

\begin{remark} {\rm
	Using that $ u^{\varepsilon}\rightrightarrows u$, $v^{\varepsilon}\rightrightarrows v$ and that $u^{\eps}\geq v^{\eps}$ we immediately obtain 
	\begin{equation}
	u(y)\geq v(y) \qquad \mbox{for all} \ y\in\R^N.
	\end{equation}
	}
\end{remark}

\begin{remark} {\rm
	We can write the system \eqref{ED1} as 
	\begin{equation}
	\label{ED12}
	\left\lbrace 
	\begin{array}{ll}
	\displaystyle \min\Big\{ -\Delta_p^1 u(x)+h_1(x),(u-v)(x)\Big\}=0, \quad & x\in\Om, \\[10pt]
	\displaystyle \max\Big\{ -\Delta_q^1 v(x)-h_2(x),(v-u)(x)\Big\}=0, \quad & x\in\Om, \\[10pt]
	u(x)=f(x), \quad & x\in\partial\Om, \\[10pt]
	v(x)=g(x), \quad & x\in\partial\Om.
	\end{array}
	\right.
	\end{equation}
	Here the first equation says that $u$ is a solution to the obstacle problem for the $p$-Laplacian with $v$ as a obstacle and boundary datum $f$, and the second equation says that $v$ is a solution to the obstacle problem for the $q$-Laplacian with $u$ as a obstacle from above and boundary datum $g$.
	
	This formulation correspond to a two membrane problem in which the membranes
	are clamped on the boundary of the domain and each membrane acts as an obstacle for the other. }
\end{remark}

\begin{remark} {\rm
Since in general there is no uniqueness for the two membranes problem
we can only show convergence taking a sequence $\varepsilon_j \to 0$ using 
a compactness argument.
}
\end{remark}

Let us briefly comment on the main difficulties that appear in the proof
of this result. To show that the (DPP) has a solution we argue using monotonicity arguments in the spirit of Perron's method (a solutions is obtained as the 
supremum of subsolutions).  Once we proved existence of a solution to the (DPP)
we use this solution to construct quasioptimal strategies for the players and show that the
game has a value that coincides with a solution to the (DPP) (this fact implies uniqueness for solutions to the (DPP)). At this point we want to mention that it is crucial the rule that forces to play one round of the game when one of the players decides to change boards. If one changes boards without playing a round in the other board the game may never end (and even if we penalize games that never end it is not clear that the game has a value). See \cite{PSSW} for an example of a 
Tug-of-War game that does not have a value.
After proving existence and uniqueness for the (DPP) and the existence of a value for the game we study its behaviour as $\eps\to 0$.
Uniform convergence will follow from a variant
of Arzela-Ascoli lemma, see Lemma \ref{lem.ascoli.arzela} (this idea was used before to obtain
convergence of value functions of games, see 
several examples in \cite{BRLibro}). To this end we need that when the game starts close to the boundary in any of the two
boards any of the two players has a strategy that forces the game to
end close to the starting point in a controlled number of plays with large probability. 
For example, starting in the first board the first player may choose the strategy 
to never change boards and to point to a boundary point when the Tug-of-War game is 
played. One can show that this strategy gives the desired one-side estimate. 
However, starting in the first board, to find a strategy for Player II that achieves similar 
bounds is trickier since the player who may decide to change boards is Player I. To obtain such bounds for the terminal position and the expected number of plays in this case
is one of the main difficulties that we deal with. Once we proved uniform convergence
of the value functions we use the DPP to obtain, using the usual viscosity approach, that
the limit pair is a solution to the two membranes problem.

\subsubsection{Second game} \label{subsect-2nd}
	Let us consider a variant of the previous game in which
	the possibility of the players to change boards
	also depends on a coin toss. 
	
	This new game has the following rules: if the position of the game is at $(x_k,1)$ the players toss a fair coin (probability $1/2-1/2$), if Player I wins, he decides to play the $J_1$ game in the first board or to play $J_2$ game in the second board. On the other hand, if the winner is Player II the only option is play $J_1$ in the first board. 	
If the position is in the second board, say at $(x_k,2)$, the situation is analogous but with the roles of the players reversed, the players tosses a fair coin again, and if Player II wins, she decides between playing $J_2$ in the second board or jumping to the first board and play $J_1$, while if the Player I wins the only option is to play $J_2$ in the second board. 
Here the rules of $J_1$ and $J_2$ are exactly as before, the only thing that we changed is that the decision to change boards or not is also dependent of a fair coin toss. 

This game has associated the following (DPP):  
	 \begin{equation}
	 	\label{DPPEXT}
	 	\left\lbrace
	 	\begin{array}{ll}
	 		\displaystyle u^{\eps}(x)=\half\max\Big\{ J_1(u^{\eps})(x), J_2(v^{\eps})(x)\Big\}+\half J_1(u^{\eps})(x) , \qquad
	 		&  x \in \Omega,  \\[10pt]
	 		\displaystyle v^{\eps}(x)=\half\min\Big\{ J_1(u^{\eps})(x), J_2(v^{\eps})(x)\Big\} +\half J_2(v^{\eps})(x),
	 		&  x \in \Omega,  \\[10pt]
	 		u^{\eps}(x) = f(x), \qquad & x \in \R^{N} \backslash \Omega,  \\[10pt]
	 		v^{\eps}(x) = g(x), \qquad &  x \in \R^{N} \backslash \Omega.
	 	\end{array}
	 	\right.
	 \end{equation}
	 
	 This (DPP) also reflects the rules of the game. For instance, the first equation
	 says that with probability $1/2$ the first player decides to play $J_1$ or to change boards and play $J_2$ (hence the term $\max \{ J_1(u^{\eps})(x), J_2(v^{\eps})(x)\}$
	 appears) and with probability $1/2$ the position stays in the first board (they just play $J_1$).
 
 \begin{remark} {\rm
 		Also in this case, from the (DPP) and the condition $f\geq g$ 
		it is immediate that $$u^{\eps}(y)\geq v^{\eps}(y)$$ holds for all $y\in\R^N$.}
 \end{remark} 
	 
In this case the pair $(u^{\eps},v^{\eps})$ also converges uniformly
along a subsequence $\eps_j \to0$ to a continuous pair $(u,v)$, and this limit pair
is also a viscosity solution to the two membrane problem with an extra condition 
in the contact set. 
 
 \begin{theorem} 
	\label{teo.limite.2}
	There exists a sequence $\eps_j \to 0$ such that $(u^{\varepsilon_j}, v^{\varepsilon_j})$ converges to a pair of continuous functions $(u,v)$,
	$$
	u^{\varepsilon_j}\rightrightarrows u, \qquad v^{\varepsilon_j}\rightrightarrows v
	$$
	uniformly in $\overline{\Omega}$. 
	The limit pair is a viscosity solution to the two membrane system
	with the two different $p-$Laplacians, \eqref{ED1}, with 
	$p$ and $q$ given by \eqref{p,q}. Moreover, the following extra condition 
	\begin{equation} \label{extra-cond.intro}
	 \big(-\Delta_{p}^{1}u(x)+ h_1(x)\big)+\big(-\Delta_q^1 v(x)- h_2(x)\big) =0, \qquad \ x\in \Om,
	\end{equation}
	holds.
\end{theorem}

 \begin{remark} {\rm
  Let us observe that the extra condition \eqref{extra-cond.intro}
  trivially holds in $\{u>v\}$ (since $u$ and $v$
  solve $-\Delta_{p}^{1}u(x)+ h_1(x)=0$ and $-\Delta_q^1 v(x)- h_2(x)=0$ respectively). 
  Then, this extra condition gives us new information in the contact set $\{u=v\}$.
  Notice that this extra condition (the sum of the two equations is equal to zero) is similar to the one that appears in \cite{Caffa1}, c.f. \eqref{extra-caffa}, but it is not the same as the one assumed in \cite{Caffa2} to obtain regularity of the solutions, see \eqref{cond-caffaVivas},
  since for the normalized $p-$Laplacian it is not true that 
  $-\Delta_{p}^{1}u(x) = \Delta_q^1 (- u) (x)$ unless $p=q$. }
\end{remark}

The paper is organized as follows: In Section \ref{sect-first-game} we analyze the first game;
in subsection
\ref{sect-DPP} we prove that the game has
a value and that this value is the unique solution to the (DPP).
The proof of Theorem \ref{teo.limite} is divided into subsections \ref{sect-conv} and  \ref{sect-lim-viscoso}. 
In the
first one we prove uniform convergence along a subsequence and in the second we
show that the uniform limit is a viscosity solution to the PDE system \eqref{ED1}.

In Section \ref{sect-second-game} we include a brief description of the analysis for the second game (the arguments used to show uniform convergence are quite similar). Here we focus on the details needed to show that we obtain an extra condition in the contact set. 

Finally, in Section \ref{sect-remarks} we include some remarks and comment on possible
extensions of our results.

\section{Analysis for the first game} \label{sect-first-game}

\subsection{Existence and uniqueness for the (DPP)}\label{sect-DPP}


In this section we first prove that there is a solution to the (DPP), \eqref{DPP}, next we show that the existence of
a solution to the (DPP) implies that the game has a value (it allows us to find quasi optimal strategies
for the players) and at the end we obtain the uniqueness of solutions to the DPP. 

To show existence of a solution to the (DPP) we use a variant of Perron's method
(that is, a solution can be obtained as supremum of subsolutions).

Let us consider the set of functions 
\begin{equation}
\mathscr{A}=\Big\{ (\ul{u}^{\eps},\ul{v}^{\eps}) \ : \ \mbox{ $\ul{u}^{\eps},\ul{v}^{\eps}$ are	bounded functions such that \eqref{DPPSup} holds} \Big\}
\end{equation}
with
\begin{equation}
\label{DPPSup}
\left\lbrace
\begin{array}{ll}
\displaystyle u^{\eps}(x)\leq \max\Big\{ J_1(u^{\eps})(x), J_2(v^{\eps})(x)\Big\} , \qquad
&  x \in \Omega,  \\[10pt]
\displaystyle v^{\eps}(x)\leq \min\Big\{ J_1(u^{\eps})(x), J_2(v^{\eps})(x)\Big\}, &   x \in \Omega, \\[10pt]
u^{\eps}(x) \leq f(x), \qquad & x \in \R^{N} \backslash \Omega,  \\[10pt]
v^{\eps}(x) \leq g(x), \qquad &  x \in \R^{N} \backslash \Omega.
\end{array}
\right.
\end{equation}

Notice that \eqref{DPPSup} is just the (DPP) with inequalities that say that $(\ul{u}^{\eps},\ul{v}^{\eps})$ is a subsolution to the (DPP) \eqref{DPP}. For the precise definition of sub and super solutions to (DPP) systems
we refer to \cite{nosotros,nosotros2}.

Let us begin proving that $\mathscr{A}$ is non-empty. To this end we introduce an auxiliary function. As $\Om\subset \R^N$ is bounded there exists $R>0$ such that $\Omega  \subset \subset B_R(0) \setminus \{0\}$ (without loss of generality we may assume that $0\not\in \Omega$). 
Consider the function
\begin{equation}
	z_0(x)=\left\lbrace
	\begin{array}{ll}
		\displaystyle 2K(|x|^2-R)-M \qquad & \ \mbox{if } \ x \in \ol{B_R(0)},  \\[7pt]
		\displaystyle -M  \qquad & \ \mbox{if } \ x \in \mathbb{R}^N \setminus B_R(0).
	\end{array}
	\right.
\end{equation}
This function have the following properties: The function $z_0$ is ${C}^2(\Om)$ and it holds,
for $x\in B_R(0) \setminus \{0\}$,
$$
\Delta z_0(x)= (z_0)_{rr} + \frac{N-1}{r} (z_0)_r =
4K + \frac{N-1}{r} 4K r = 4K + 4K (N-1)
$$
and 
$$
\Delta^1_{\infty} z_0(x)= (z_0)_{rr}(x) =4K.
$$
Notice that when we compute the infinity Laplace
operator of $z_0$ we have to pay special care at the origin (where the gradient
of $z_0$
vanishes), since the operator is not well-defined there. In doing this we use that 
$z_0$ is a radial function and compute the infinity Laplacian in the 
classical sense at points in $B_R(0) \setminus \{0\}$ (where the gradient
does not vanish).

Then, we get
\begin{equation}
	\begin{array}{l}
		\displaystyle \Delta^1_p z_0 (x)=\frac{\alpha_1}{2}(4K)+\frac{(1-\alpha_1)}{2(N+2)}\big(4K+4K(N-1)\big)\geq 4K\\[8pt]
		\qquad \mbox{and} \\[8pt]
		\displaystyle \Delta^1_q z_0 (x)=\frac{\alpha_2}{2}(4K)+\frac{(1-\alpha_2)}{2(N+2)}\big(4K+4K(N-1)\big)\geq 4K.
	\end{array}	
\end{equation}

We are ready to prove the first lemma.

\begin{Lemma}
For $\eps$ small enough it holds that $\mathscr{A}\neq\emptyset$.
\end{Lemma}

\begin{proof}
	We consider $z_0$ with the constants $$K=\max\{\lVert h_1\rVert_{\infty},\lVert h_2\rVert_{\infty}\}+1$$ and $$M=\max\{\lVert f\rVert_{\infty},\lVert g\rVert_{\infty}\}+1,$$
and claim that $$(z_0,z_0)\in\mathscr{A}.$$

Let us prove this claim. First, we observe that the inequality \eqref{DPPSup} holds for 
$x\in\R^{N}\backslash\Om$. Then, we are left to prove that for $x\in\Omega$ it holds that 
$$
z_0(x)\leq \min\Big\{J_1(z_0)(x),J_2(z_0)(x)\Big\}.
$$
That is, we aim to show that
\begin{equation}
	\label{rec1}
	0 \leq \min\Big\{J_1(z_0)(x)-z_0(x),J_2(z_0)(x)-z_0(x)\Big\}.
\end{equation}
Using Taylor's expansions we obtain
\begin{equation}
	\begin{array}{l}
		\displaystyle	J_1(z_0)(x)-z_0(x)\\[8pt]
		\displaystyle	\qquad = \alpha_1\big[\half\sup_{y \in B_{\eps}(x)}(z_0(y)-z_0(x))+\half\inf_{y \in B_{\eps}(x)}(z_0(y)-z_0(x))\big]\\[8pt]
		\displaystyle \qquad \qquad +(1-\alpha_1)\kint_{B_{\eps}(x)}(z_0(y)-z_0(x))dy-\eps^2 h_1(x)\\[8pt]
		\displaystyle \qquad =\Big(\frac{\alpha_1}{2}\Delta^1_{\infty}z_0(x)+\frac{(1-\alpha_1)}{2(N+2)}\Delta z_0(x)\Big)\eps^2-\eps^2 h_1(x) +o(\eps^2).
	\end{array}
\end{equation}
Analogously, 
\begin{equation}
	\begin{array}{l}
		\displaystyle	J_2(z_0)(x)-z_0(x)=\Big(\frac{\alpha_2}{2}\Delta^1_{\infty}z_0(x)+\frac{(1-\alpha_2)}{2(N+2)}\Delta z_0(x)\Big)\eps^2+\eps^2 h_2(x) +o(\eps^2).
	\end{array}
\end{equation}
If we come back to \eqref{rec1} and we divide by $\eps^2$ we obtain
\begin{equation}
	\label{maxdosop1}
	0\leq\min\Big\{\Delta^1_p z_0(x)-h_1(x),\Delta^1_q z_0(x)+h_2(x)\Big\}.
\end{equation}
Using the properties of $z_0$ we have
\begin{equation}
	\Delta^1_p z_0(x)-h_1(x)\geq 3K \quad  \mbox{and} \quad \Delta^1_q z_0(x)+h_2(x)\geq 3K.
\end{equation}
Thus, the inequality \eqref{maxdosop1} holds for $\eps$ small enough. 
This ends the proof.
\end{proof}

\begin{remark} {\rm
	We can define a different auxiliary function $z^{\eps}$ as the solution to the following 
	problem
	\begin{equation} \label{ju}
	\left\lbrace
	\begin{array}{ll}
	\displaystyle z^{\eps}(x)=\min_{i\in\{1,2\}}\Big\{\alpha_i\big(\half\sup_{y \in B_{\eps}(x)}z^{\eps}(y)+\half\inf_{y \in B_{\eps}(x)}z^{\eps}(y)\big)+(1-\alpha_i)\kint_{B_{\eps}(x)}z^{\eps}(y)dy\Big\}-\eps^2 K, & x\in\Om, \\[10pt]
	\displaystyle z^{\eps}(x)=-M, & x\notin\Om.
	\end{array}
	\right.
	\end{equation}
	The existence of this function is given in \cite{BlancPR} (see Theorem 1.5). 
	In fact, \eqref{ju} is the (DPP) that corresponds to a game in which one player
	(the one that wants to minimize the expected payoff) chooses the coin that 
	decides the game to play between Tug-of-War and a random walk. 
	
	If we argue as before we can prove that $(z^{\eps},z^{\eps})\in\mathscr{A}$.}
\end{remark}

Now our goal is to show that the functions $(\ul{u}^{\eps},\ul{v}^{\eps})\in\A$ are uniformly bounded.
To prove this fact we will need some lemmas. 
Let us consider the function $w_0=-z_0$. 
 This function have the following properties: 
	$$
	\Delta w_0(x)= -\Delta z_0 = -4K - 4K(N-1)
	$$
	and 
	$$
	\Delta^1_{\infty} w_0(x)=-\Delta^1_{\infty} w_0(x)=-4K.
	$$
	Then, we have
	\begin{equation}
	\begin{array}{l}
	\displaystyle \Delta^1_p w_0 (x)=\frac{\alpha_1}{2}(-4K)+\frac{(1-\alpha_1)}{2(N+2)}\big(-4K-4K(N-1)\big)\leq -4K\\[8pt]
	 \qquad \mbox{and} \\[8pt]
	\displaystyle \Delta^1_q w_0 (x)=\frac{\alpha_2}{2}(-4K)+\frac{(1-\alpha_2)}{2(N+2)}\big(-4K-4K(N-1)\big)\leq -4K.
	\end{array}	
	\end{equation}

Let us prove a technical lemma.

\begin{Lemma} \label{lema.1}
	\label{w0} Given $K=\max \{\lVert h_1\rVert_{\infty},\lVert h_1\rVert_{\infty}\}+1$ and $M =\max\Big\{\lVert f\rVert_{\infty},\lVert g\rVert_{\infty} \Big\}+1$, there exists $\eps_0>0$ such that
	the function $w_0$ verifies 
	\begin{equation}
	\label{DPP2}
	\left\lbrace
	\begin{array}{ll}
	\displaystyle w_0 (x) \geq \max 
	\Big\{J_1(w_0)(x),J_2(w_0)(x)\Big\}+K\eps^2, \qquad &  \ x \in \Omega,  \\[12pt]
	\displaystyle w_0(x)\geq M ,  \qquad & \ x \in\R^N\backslash \Omega,
	\end{array}
	\right.
	\end{equation}
	for every $\eps < \eps_0$.
\end{Lemma}

\begin{proof} First, let us observe that the inequality $w^{\eps}(x)\geq M$ holds for 
	$x\in\R^{N}\backslash\Om$ when $\tilde{M}$ is large enough. Then, we are left to prove that for $x\in\Omega$ it holds that 
	$$
	w_0(x)\geq \max\Big\{J_1(w_0)(x),J_2(w_0)(x)\Big\}+K\eps^2.
	$$
	That is, 
	\begin{equation}
	\label{rec}
	0 \geq \max\Big\{J_1(w_0)(x)-w_0(x),J_2(w_0)(x)-w_0(x)\Big\}+K\eps^2.
	\end{equation}
	Using Taylor's expansions we obtain
	\begin{equation}
	\begin{array}{l}
	\displaystyle	J_1(w_0)(x)-w_0(x)\\[8pt]
	\displaystyle	\qquad = \alpha_1\big[\half\sup_{y \in B_{\eps}(x)}(w_0(y)-w_0(x))+\half\inf_{y \in B_{\eps}(x)}(w_0(y)-w_0(x))\big]\\[8pt]
	\displaystyle \qquad \qquad +(1-\alpha_1)\kint_{B_{\eps}(x)}(w_0(y)-w_0(x))dy+\eps^2 h_1(x)\\[8pt]
	\displaystyle \qquad =\Big(\frac{\alpha_1}{2}\Delta^1_{\infty}w_0(x)+\frac{(1-\alpha_1)}{2(N+2)}\Delta w_0(x)\Big)\eps^2-\eps^2 h_1(x) +o(\eps^2).
	\end{array}
	\end{equation}
	Analogously, 
	\begin{equation}
	\begin{array}{l}
	\displaystyle	J_2(w_0)(x)-w_0(x)=\Big(\frac{\alpha_2}{2}\Delta^1_{\infty}w_0(x)+\frac{(1-\alpha_2)}{2(N+2)}\Delta w_0(x)\Big)\eps^2-\eps^2 h_2(x) +o(\eps^2)
	\end{array}
	\end{equation}
	If we come back to \eqref{rec} and we divide by $\eps^2$ we get
	\begin{equation}
	\label{maxdosop}
	0\geq\max\Big\{\Delta^1_p w_0(x)-h_1(x),\Delta^1_q w_0(x)+h_2(x)\Big\}+K.
	\end{equation}
	Using the properties of $w_0$ we arrive to
	\begin{equation}
	\Delta^1_p w_0(x)-h_1(x)\leq -4K \quad \mbox{and} \quad \Delta^1_q w_0(x)+h_2(x)\leq -3K.
	\end{equation}
	Thus, the inequality \eqref{maxdosop} holds for $\eps$ small enough. 
	This ends the proof.
\end{proof}

Our next result says that in fact the subsolutions
to the (DPP) (pairs $(\ul{u},\ul{v})\in\mathscr{A}$) are indeed bounded by $w_0$.  This shows that functions
in $\mathscr{A}$ are uniformly bounded. 
From the proof of the following result one can obtain a comparison principle for the (DPP). 

\begin{Lemma}
	\label{subsup}
	Let $(\ul{u}^{\eps},\ul{v}^{\eps})\in\mathscr{A}$ (bounded subsolutions to the (DPP) \eqref{DPP})
	and  $w^{\eps}$ a function that verifies  \eqref{DPP2}, that is, 
	$$
	\left\lbrace
	\begin{array}{ll}
	\displaystyle w^{\eps} (x) \geq \max 
	\Big\{J_1(w^{\eps})(x),J_2(w^{\eps})(x)\Big\}+K\eps^2, \qquad &  \ x \in \Omega,  \\[12pt]
	\displaystyle w^{\eps} (x)\geq M ,  \qquad & \ x \in\R^N\backslash \Omega.
	\end{array}
	\right.
	$$
	Then, it holds that
	$$
	\ul{u}^{\eps}(x)\leq w^{\eps} (x) \quad \mbox{ and } \quad \ul{v}^{\eps}(x)\leq w^{\eps}(x),
	\qquad x\in \mathbb{R}^N.
	$$
\end{Lemma}

\begin{proof}
	We argue by contradiction. Assume that 
	$$
	\max \Big\{\sup (\ul{u}^{\eps}-w^{\eps}),  \sup(\ul{v}^{\eps}-w^{\eps}) \Big\}=\theta > 0.
	$$
	It is clear that
	$$
	\ul{u}^{\eps}(x)\leq M \leq w^{\eps} (x), \qquad \ul{v}^{\eps}(x) \leq M \leq w^{\eps}(x) .
	$$
	for $x\not \in \Omega$. Therefore, we have to concentrate in what 
	happens inside $\Omega$. We divide the proof into two cases.
	
	\textbf{First case.} Assume that $$\sup(\ul{v}^{\eps}-w^{\eps})=\theta. $$ 
	Given $n\in\mathbb{N}$ let $x_n\in\Om$ be such that 
	$$
	\theta - \dfrac{1}{n} < (\ul{v}^{\eps}-w^{\eps})(x_n).
	$$
	
	We use the inequalities verified 
	by the involved functions to obtain	
	$$
	\begin{array}{l}
	\displaystyle 
	\theta - \dfrac{1}{n}<(\ul{v}^{\eps}-w^{\eps})(x_n)\leq J_2(\ul{v}^{\eps})(x_n)-J_2(w^{\eps})(x_n) 
	\\[10pt]
	\displaystyle =\alpha_2\big[\half\sup_{y \in B_{\eps}(x_n)}\ul{v}^{\eps}(y)+\half\inf_{y \in B_{\eps}(x_n)}\ul{v}^{\eps}(y)-\half\sup_{y \in B_{\eps}(x_n)}w^{\eps}(y)-\half\inf_{y \in B_{\eps}(x_n)}w^{\eps}(y)\big]\\[10pt]
	\displaystyle \qquad +(1-\alpha_2)\kint_{B_{\eps}(x_n)}(\ul{v}^{\eps}-w^{\eps})(y)dy+\eps^2 h_2(x_n)-\eps^2 K\\[10pt]
	\displaystyle \leq \alpha_2\big[\half\sup_{y \in B_{\eps}(x_n)}\ul{v}^{\eps}(y)-\half\sup_{y \in B_{\eps}(x_n)}w^{\eps}(y)+\half\sup_{y \in B_{\eps}(x_n)}(\ul{v}^{\eps}-w^{\eps})(y)\big]\\[10pt]
	\displaystyle \qquad +(1-\alpha_2)\kint_{B_{\eps}(x_n)}(\ul{v}^{\eps}-w^{\eps})(y)dy+\eps^2h_2(x_n)-\eps^2 K.
	\end{array}
	$$
	Now we use that  
	$$
	\sup_{y \in B_{\eps}(x_n)}(\ul{v}^{\eps}-w^{\eps})(y)\leq\theta \quad \mbox{and} \quad \kint_{B_{\eps}(x_n)}(\ul{v}-w)(y)dy\leq \theta
	$$ 
	to arrive to 
	$$
	\theta-\dfrac{2}{n\alpha_2}<\sup_{y \in B_{\eps}(x_n)}\ul{v}^{\eps}(y)-\sup_{y \in B_{\eps}(x_n)}w^{\eps}(y)+(2h_2(x_n)-2K)\frac{\eps^2}{\alpha_2}.
	$$
	Take $y_n\in B_{\eps}(x_n)$ such that $$\sup_{y \in B_{\eps}(x_n)}\ul{v}^{\eps}(y)-\dfrac{1}{n}<\ul{v}^{\eps}(y_n).$$
	Then, we get
	\begin{equation}
		\begin{array}{ll}
		\displaystyle	\theta-\dfrac{2}{n\alpha_2}<\ul{v}^{\eps}(y_n)+\dfrac{1}{n}-\sup_{y \in B_{\eps}(x_{\eps})}w^{\eps}(y)(2h_2(x_n)-2K\eps^2)\\
			\displaystyle\leq \ul{v}^{\eps}(y_n)+\dfrac{1}{n}- w^{\eps}(y_n)-\frac{\eps^2}{\alpha_2}.
		\end{array}
	\end{equation}
	Here we use that $h_2(x)-K\leq -1$. Hence
	$$
	\theta-\dfrac{2-\alpha_2}{n\alpha_2}+\frac{\eps^2}{\alpha_2}< (\ul{v}^{\eps}- w^{\eps})(y_n)\leq \theta 
	$$
	that leads to a contradiction if $n\in\mathbb{N}$ is large enough
	in order to have $$-\dfrac{2-\alpha_2}{n\alpha_2}+\frac{\eps^2}{\alpha_2}>0$$ since in this case we 
	obtain
	$$
	\theta<\theta-\dfrac{2-\alpha_2}{n\alpha_2}+\frac{\eps^2}{\alpha_2}< (\ul{u}^{\eps}- w^{\eps})(y_n)\leq \theta.
	$$
	This ends the proof in the first case.
	
	\textbf{Second case.} Assume that $$\sup(\ul{u}^{\eps}-w^{\eps})=\theta.$$ 
	In this case we take again 
	a sequence $x_n\in\Om$ such that $$\theta-\dfrac{1}{n}<(\ul{u}^{\eps}-w^{\eps})(x_n).$$
	Let us assume first that $$\max \Big\{J_1(\ul{u}^{\eps})(x_n),J_2(\ul{v}^{\eps})(x_n)\Big\}=J_2(\ul{v}^{\eps})(x_n),$$
	then we obtain 
	\begin{equation}
	(\ul{u}^{\eps}-w^{\eps})(x_n)\leq J_2(\ul{v}^{\eps})(x_n)-J_2(w^{\eps})(x_n),
	\end{equation} 
	we are again in the first case and we arrive to a contradiction arguing as before.
	
	Finally, let us assume that $$\max \Big\{J_1(\ul{u}^{\eps})(x_n),J_2(\ul{v}^{\eps})(x_n)\Big\}=J_1(\ul{u}^{\eps})(x_n),$$
	then we obtain 
	\begin{equation}
	(\ul{u}^{\eps}-w^{\eps})(x_n)\leq J_1(\ul{u}^{\eps})(x_n)-J_1(w^{\eps})(x_n),
	\end{equation}
	and if we argue as in the first case we arrive to a contradiction. This ends the proof. 
\end{proof}

Now, using that $w_0$ is continuous $\R^N$ and hence bounded in the ball $\overline{B_R}$, 
we can deduce that there exists a constant $\Lambda>0$ 
that depends on the data $f$, $g$, $h$ and the domain $\Omega$ such that $w_0(x)\leq\Lambda$. 
Then, using the previous lemmas we obtain a uniform bound for 
functions in $\A$. 

\begin{theorem} \label{th.cota.uniforme}
	There exists a constant $\Lambda>0$ that depends on $f$,$g$,$h$ and $\Omega$ 
	such that for every $(\ul{u}^{\eps},\ul{v}^{\eps})\in \A$ it holds that
	$$
	\ul{u}^{\eps} (x) \leq \Lambda \qquad \mbox{and} \qquad \ul{v}^{\eps}(x) \leq \Lambda,
	$$
	for every $x \in \mathbb{R}^N$ and every $\eps \leq \eps_0$ (here $\eps_0$ is given
	by Lemma \ref{lema.1}). 
\end{theorem}

\begin{proof}
	It follows from Lemma \ref{lema.1}, Lemma \ref{subsup} and the boundedness of $w_0$.
\end{proof}

With this result at hand we can define for $x \in \mathbb{R}^N$,
\begin{equation} \label{eq.supremos}
u^{\eps}(x)=\sup_{(\ul{u}^{\eps},\ul{v}^{\eps})\in\A}\ul{u}^{\eps}(x) \qquad \mbox{and} \qquad
v^{\eps}(x)=\sup_{(\ul{u}^{\eps},\ul{v}^{\eps})\in\A}\ul{v}^{\eps}(x).
\end{equation}

The previous result, Theorem \ref{th.cota.uniforme}, gives that these two functions
$u^{\eps}$ and $v^{\eps}$ are well defined and bounded. It turns out that they are
a solution to the (DPP). 

\begin{theorem} \label{existencia.DPP}
	The pair of functions $(u^{\eps},v^{\eps})$ is a solution to the (DPP), \eqref{DPP}. 
\end{theorem}

\begin{proof}
	First, let us show that $(u^{\eps},v^{\eps})\in\A$. Given $(\ul{u}^{\eps},\ul{v}^{\eps})\in\A$ and $x\in \Omega$ 
	we have that
	$$
	\ul{u}^{\eps}(x)\leq\max\Big\{J_1(\ul{u}^{\eps})(x),J_2(\ul{v}^{\eps})(x)\Big\}. 
	$$
	Taking supremum for $(\ul{u}^{\eps}, \ul{v}^{\eps}) \in \A$ we obtain
	$$
	\ul{u}^{\eps}(x)\leq\max\Big\{J_1(u^{\eps})(x),J_2(v^{\eps})(x)\Big\}
	$$	
	and hence (taking supremum in the left hand side) we conclude that 
	$$
	u^{\eps}(x)\leq\max\Big\{J_1(u^{\eps})(x),J_2(v^{\eps})(x)\Big\}
	$$			
	
	An analogous computation for the second equation shows that $v^{\eps}$
	verifies 
	$$
	v^{\eps}(x)\leq \min\Big\{J_1(u^{\eps})(x),J_2(v^{\eps})(x)\Big\} 
	$$
	for $x \in \Omega$. Finally, as 
	$\ul{u}^{\eps}(x) \leq f(x)$ and $\ul{v}^{\eps}(x) \leq g(x)$ for $x \in \R^{N} \backslash \Omega$,  
	taking supremum we obtain 
	$u^{\eps}(x) \leq f(x)$ and $v^{\eps}(x) \leq g(x)$ for $x \in \R^{N} \backslash \Omega$ and we conclude that
	$(u^{\eps},v^{\eps})\in\A$. 
	
	We have a set of inequalities for the pair $(u^{\eps},v^{\eps})\in\A$. To show that the pair is
	indeed a solution to the DPP we have to show that they are in fact equalities. 
	To prove this fact we argue by contradiction. Assume that we have a strict 
	inequality for some $x_0\in\R^N$. If $ x_0\in \R^N\backslash\Om$
	and we have $u^{\eps}(x_0) < f(x_0)$ we will reach a contradiction
	considering 
	$$
	u^{\eps}_0(x)=\left\lbrace
	\begin{array}{ll}
	\displaystyle u^{\eps}(x) \qquad & \ \mbox{if } \ x \neq x_0,  \\[7pt]
	\displaystyle u^{\eps}(x_0)+\delta \qquad & \ \mbox{if } \ x=x_0,
	\end{array}
	\right.
	$$
	with $\delta>0$ small enough such that 
	$u^{\eps}(x_0)+\delta < f(x_0)$. Indeed, 
	one can check that the pair $(u^{\eps}_0, v^{\eps})$ belongs to $\A$ but at $x_0$ we have
	$u^{\eps}_0 (x_0) = u^{\eps}(x_0)+\delta > u^{\eps}(x_0) = \sup_{\A} \ul{u}^{\eps} (x_0)$,
	a contradiction. A similar argument can be used when $ x_0\in \R^N\backslash\Om$
	and we have $v^{\eps}(x_0) < g(x_0)$. We conclude that $u^{\eps}(x) = f(x)$ and $v^{\eps}(x) = g(x)$ for every $x\in \R^N\backslash\Om$.
	
	Now, let us assume that the point at which we have a strict inequality is inside $\Omega$, 
	$x_0\in\Om$. First, assume that we have
	$$
	u^{\eps}(x_0)<\max\Big\{J_1(u^{\eps})(x_0),J_2(v^{\eps})(x_0)\Big\}.
	$$	
	Let us consider
	$$
	\max\Big\{J_1(u^{\eps})(x_0),J_2(v^{\eps})(x_0)\Big\}-u^{\eps}(x_0)=\delta>0,
	$$	
	and, as before, the function
	$$
	u^{\eps}_0(x)=\left\lbrace
	\begin{array}{ll}
	\displaystyle u^{\eps}(x) \qquad & \ \mbox{if } \ x \neq x_0,  \\[7pt]
	\displaystyle u^{\eps}(x_0)+\frac{\delta}{2} \qquad & \ \mbox{if } \ x=x_0.  
	\end{array}
	\right.
	$$	
	Then, we have
	$$
	u^{\eps}_0(x_0)=u^{\eps}(x_0)+\frac{\delta}{2}<\max\Big\{J_1(u^{\eps})(x_0),J_2(v^{\eps})(x_0)\Big\},
	$$	
	At other points $x\in \Omega$ we also have 
	$$
	u^{\eps}_0(x) \leq \max\Big\{J_1(u^{\eps})(x),J_2(v^{\eps})(x)\Big\}\leq\max\Big\{J_1(u_0^{\eps})(x),J_2(v^{\eps})(x)\Big\}.
	$$
	Finally, concerning $v^\eps$ we get (at any point $x\in \Omega$),
	$$
	v^{\eps}(x)\leq \min\Big\{J_1(u_0^{\eps})(x),J_2(v^{\eps})(x)\Big\} .
	$$
	Hence, we have that the pair $(u^{\eps}_0, v^{\eps})$ belongs to $\A$, getting a contradiction as before.
	since $u_{0}^{\eps}(x_0) >u^{\eps} (x_0)$.

	Analogously, one can deal with the case in which $x_0 \in \Omega$ and  
	$$
	v^{\eps}(x_0)<\min\Big\{J_1(u^{\eps})(x_0),J_2(v^{\eps})(x_0)\Big\}.
	$$
	The proof is finished. 
\end{proof}

\begin{corol} \label{corl-acot}
	There exists a constant $\Lambda>0$ such that 
	\begin{equation}
	|u^{\eps}(x)|<\Lambda \qquad \mbox{and} \qquad |v^{\eps}(x)|<\Lambda 
	\end{equation} 
	for all $x\in\R^{N}$.
\end{corol}

\begin{proof}
	Every solution to the (DPP) belongs to $\mathcal{A}$. Hence the result follows from
	 Theorem \ref{th.cota.uniforme}. 
\end{proof}

Now, for completeness, we include the precise statement of the 
Optional Stopping Theorem (a key tool from probability theory that we will use in what follows).

{\bf The Optional Stopping Theorem.}
We briefly recall (see \cite{Williams}) that a sequence of random variables
$\{M_{k}\}_{k\geq 1}$ is a supermartingale (submartingale) if
$$ \E[M_{k+1}\arrowvert M_{0},M_{1},...,M_{k}]\leq M_{k} \ \ (\geq).$$
Then, the Optional Stopping Theorem, that we will call {\it (OSTh)} in what follows, says:
given $\tau$ a stopping time such that one of the following conditions hold,
\begin{itemize}
	\item[(a)] The stopping time $\tau$ is bounded almost surely;
	\item[(b)] It holds that $\E[\tau]<\infty$ and there exists a constant $c>0$ such that $$\E[M_{k+1}-M_{k}\arrowvert M_{0},...,M_{k}]\leq c;$$
	\item[(c)] There exists a constant $c>0$ such that $|M_{\min \{\tau,k\}}|\leq c$ almost surely for every $k$.
\end{itemize}
Then 
$$ \E[M_{\tau}]\leq \E [M_{0}] \ \ (\geq)$$
if $\{M_{k}\}_{k\geq 0}$ is a supermartingale (submartingale).
For the proof of this classical result we refer to \cite{Doob,Williams}. 

Let us finish this section proving the following theorem.

\begin{theorem} \label{teo.5}
	The functions $u^{\eps}$ and $v^{\eps}$ that verifies the (DPP) \eqref{DPP} are the functions that give the value of the game in \eqref{value}. This means, the function
	\begin{equation}
	w^{\eps}(x,j)=\inf_{S_{II}}\sup_{S_{I}}\E_{S_{I},S_{II}}^{(x,j)}[\mbox{total payoff}]=\sup_{S_{I}}\inf_{S_{II}}\E_{S_{I},S_{II}}^{(x,j)}[\mbox{total payoff}]
	\end{equation}
	verifies that $$w^{\eps}(x,1)=u^{\eps}(x)$$ and $$w^{\eps}(x,2)=v^{\eps}(x)$$ 
	for any pair $(u^\eps,v^\eps)$ that solves the (DPP), that is, for any pair that verifies
	\begin{equation}
	\left\lbrace
	\begin{array}{ll}
	\displaystyle u^{\eps}(x)=\max\Big\{ J_1(u^{\eps})(x), J_2(v^{\eps})(x)\Big\},
	\qquad 
	&  x \in \Omega,  \\[10pt]
	\displaystyle u^{\eps}(x)=\min\Big\{ J_1(u^{\eps})(x), J_2(v^{\eps})(x)\Big\}, 
	&  x \in \Omega,  \\[10pt]
	u^{\eps}(x) = f(x), \qquad & x \in \R^{N} \backslash \Omega,  \\[10pt]
	v^{\eps}(x) = g(x), \qquad &  x \in \R^{N} \backslash \Omega.
	\end{array}
	\right.
	\end{equation}
\end{theorem}

\begin{proof}
	Fix $\delta>0$. Assume that we start at a point in the first board, $(x_0,1)$. Then, we choose a strategy $S_{I}^{\ast}$ for Player I using the solution to the (DPP) \eqref{DPP} as follows:
	Whenever $j_k =1$ Player I decides to stay in the first board if
	$$
	\max \Big\{J_1(u^{\eps})(x_k),
	J_2(v^{\eps})(x_k)\Big\} =J_1(u^{\eps})(x_k) ,
	$$
	and in this case Player I chooses a point 
		\begin{equation}
	x_{k+1}^I=S_{I}^{\ast}\big((x_0,j_0),\dots,(
	x_k,j_k)\big) \quad \mbox{such that} \sup_{y \in B_{\eps}(x_k)}u^{\eps}(y)-\frac{\delta}{2^{k+1}}\leq u^{\eps}(x_{k+1}^I).
	\end{equation}
	
	On the other hand, Player I decides to jump to the second board if
	$$
	\max \Big\{J_1(u^{\eps})(x_k),
	J_2(v^{\eps})(x_k)\Big\} =J_2(v^{\eps})(x_k) ,
	$$
	and in this case Player I chooses a point 
	\begin{equation}
	x_{k+1}^I=S_{I}^{\ast}\big((x_0,j_0),\dots,(
	x_k,j_k)\big) \quad \mbox{such that} \sup_{y \in B_{\eps}(x_k)}v^{\eps}(y)-\frac{\delta}{2^{k+1}}\leq v^{\eps}(x_{k+1}^I).
	\end{equation}
	
	Given this strategy for Player I and any strategy for Player II, we consider the sequence of random variables
	\begin{equation}
	M_k=w^{\eps}(x_k,j_k)-\eps^2\sum_{l=0}^{k-1}\big(h_1(x_l)\chi_{\{j=1\}}(j_{l+1})-h_2(x_l)\chi_{\{j=2\}}(j_{l+1})\big)-\frac{\delta}{2^k}.
	\end{equation}
	where $w^{\eps}(x_k,1)=u^{\eps}(x_k)$, $w^{\eps}(x_k,2)=v^{\eps}(x_k)$ and
	\begin{equation}
	\chi_{\{j=i\}}(j)=\left\lbrace
	\begin{array}{ll}
	\displaystyle 1 & j=i ,\\[8pt]
	\displaystyle 0 & j\neq i.
	\end{array}
	\right. 
	\end{equation}
	
	Let us see that $(M_k)_{k\geq 0}$ is a \textit{submartingale}. To this end, we need to estimate
	\begin{equation}
	\E_{S_{I}^{\ast},S_{II}}^{(x_0,1)}[M_{k+1}|M_0,\dots ,M_k].
	\end{equation}
	
	Let us consider several cases.
	
	\textbf{\underline{Case 1:}} Suppose that $j_k=1$ and $j_{k+1}=1$ (that is, we stay in the first board), then
	\begin{equation}
	\begin{array}{ll}
	\displaystyle \E_{S_{I}^{\ast},S_{II}}^{(x_0,1)}[M_{k+1}|M_0,\dots ,M_k] \\[10pt]
	\displaystyle =\E_{S_{I}^{\ast},S_{II}}^{(x_0,1)}\Big[u^{\eps}(x_{k+1})-\eps^2\sum_{l=0}^{k}\big(h_1(x_l)\chi_{\{j=1\}}(j_{l+1})-h_2(x_l)\chi_{\{j=2\}}(j_{l+1})\big)-\frac{\delta}{2^{k+1}}|M_0,\dots ,M_k\Big]\\[10pt]
	\displaystyle \underbrace{=}_{j_{k+1}=1}\E_{S_{I}^{\ast},S_{II}}^{(x_0,1)}\Big[u^{\eps}(x_{k+1})-\eps^2h_1(x_k) \\[10pt]
	\qquad \quad \displaystyle -\eps^2\sum_{l=0}^{k-1}\big(h_1(x_l)\chi_{\{j=1\}}(j_{l+1})-h_2(x_l)\chi_{\{j=2\}}(j_{l+1})\big)-\frac{\delta}{2^{k+1}}|M_0,\dots ,M_k\Big]\\[10pt]
	\displaystyle = \alpha_1\big(\half u^{\eps}(x_{k+1}^{I})+\half u^{\eps}(x_{k+1}^{II})\big)+(1-\alpha_1)\kint_{B_{\eps}(x_k)}u^{\eps}(y)dy-\eps^2h_1(x_k)\\[10pt]
	\displaystyle \qquad \quad -\eps^2\sum_{l=0}^{k-1}\big(h_1(x_l)\chi_{\{j=1\}}(j_{l+1})-h_2(x_l)\chi_{\{j=2\}}(j_{l+1})\big)-\frac{\delta}{2^{k+1}}
	\end{array}
	\end{equation}
	Hence, we get
	\begin{equation}
	\begin{array}{ll}
	\displaystyle \displaystyle \E_{S_{I}^{\ast},S_{II}}^{(x_0,1)}[M_{k+1}|M_0,\dots ,M_k] 
	\\[10pt]
	\displaystyle \geq \alpha_1\big(\half \sup_{y \in B_{\eps}(x_k)}u^{\eps}(y)-\frac{\delta}{2^{k+1}}+\half\inf_{y \in B_{\eps}(x_k)}u^{\eps}(y)\big)+(1-\alpha_1)\kint_{B_{\eps}(x_k)}u^{\eps}(y)dy-\eps^2h_1(x_k)\\[10pt]
	\displaystyle \qquad -\eps^2\sum_{l=0}^{k-1}\big(h_1(x_l)\chi_{\{j=1\}}(j_{l+1})-h_2(x_l)\chi_{\{j=2\}}(j_{l+1})\big)-\frac{\delta}{2^{k+1}}
	\\[10pt]
	\displaystyle
	\geq J_1(u^{\eps})(x_k) -\eps^2\sum_{l=0}^{k-1}\big(h_1(x_l)\chi_{\{j=1\}}(j_{l+1})-h_2(x_l)\chi_{\{j=2\}}(j_{l+1})\big)-\frac{\delta}{2^k}
	\\[10pt]
	\displaystyle
	=\max\{J_1(u^{\eps})(x_k),J_2(v^{\eps})(x_k)\} -\eps^2\sum_{l=0}^{k-1}\big(h_1(x_l)\chi_{\{j=1\}}(j_{l+1})-h_2(x_l)\chi_{\{j=2\}}(j_{l+1})\big)-\frac{\delta}{2^k}
	\\[10pt]
	\displaystyle
	=u^{\eps}(x_k) -\eps^2\sum_{l=0}^{k-1}\big(h_1(x_l)\chi_{\{j=1\}}(j_{l+1})-h_2(x_l)\chi_{\{j=2\}}(j_{l+1})\big)-\frac{\delta}{2^k}=M_k.
	\end{array}
	\end{equation}
	
	\textbf{\underline{Case 2:}} Suppose that $j_k=1$ and $j_{k+1}=2$ (that is, we jump to the second board), then
	\begin{equation}
	\begin{array}{ll}
	\displaystyle \E_{S_{I}^{\ast},S_{II}}^{(x_0,1)}[M_{k+1}|M_0,\dots ,M_k]\\[10pt]
		\displaystyle =\E_{S_{I}^{\ast},S_{II}}^{(x_0,1)}\Big[v^{\eps}(x_{k+1})-\eps^2\sum_{l=0}^{k}\big(h_1(x_l)\chi_{\{j=1\}}(j_{l+1})-h_2(x_l)\chi_{\{j=2\}}(j_{l+1})\big)-\frac{\delta}{2^{k+1}}|M_0,\dots ,M_k\Big]\\[10pt]
	\displaystyle \underbrace{=}_{j_{k+1}=2}\E_{S_{I}^{\ast},S_{II}}^{(x_0,1)}\Big[v^{\eps}(x_{k+1})+\eps^2h_2(x_k) \\[10pt]
	\qquad \quad \displaystyle -\eps^2\sum_{l=0}^{k-1}\big(h_1(x_l)\chi_{\{j=1\}}(j_{l+1})-h_2(x_l)\chi_{\{j=2\}}(j_{l+1})\big)-\frac{\delta}{2^{k+1}}|M_0,\dots ,M_k\Big]
	\end{array}
	\end{equation}
	Then, we obtain
	\begin{equation}
	\begin{array}{ll}
	\displaystyle \E_{S_{I}^{\ast},S_{II}}^{(x_0,1)}[M_{k+1}|M_0,\dots ,M_k]\\[10pt]
		\displaystyle \geq \alpha_2\big(\half \sup_{y \in B_{\eps}(x_k)}v^{\eps}(y)-\frac{\delta}{2^{k+1}}+\half\inf_{y \in B_{\eps}(x_k)}v^{\eps}(y)\big)+(1-\alpha_2)\kint_{B_{\eps}(x_k)}v^{\eps}(y)dy+\eps^2h_2(x_k)\\[10pt]
	\displaystyle  \qquad -\eps^2\sum_{l=0}^{k-1}\big(h_1(x_l)\chi_{\{j=1\}}(j_{l+1})-h_2(x_l)\chi_{\{j=2\}}(j_{l+1})\big)-\frac{\delta}{2^{k+1}}\\[10pt]
	\displaystyle = J_2(v^{\eps})(x_k) -\eps^2\sum_{l=0}^{k-1}\big(h_1(x_l)\chi_{\{j=1\}}(j_{l+1})-h_2(x_l)\chi_{\{j=2\}}(j_{l+1})\big)-\frac{\delta}{2^k}\\[10pt]
	\displaystyle =\max\{J_1(u^{\eps})(x_k),J_2(v^{\eps})(x_k)\} -\eps^2\sum_{l=0}^{k-1}\big(h_1(x_l)\chi_{\{j=1\}}(j_{l+1})-h_2(x_l)\chi_{\{j=2\}}(j_{l+1})\big)-\frac{\delta}{2^k}\\[10pt]
	\displaystyle =u^{\eps}(x_k) -\eps^2\sum_{l=0}^{k-1}\big(h_1(x_l)\chi_{\{j=1\}}(j_{l+1})-h_2(x_l)\chi_{\{j=2\}}(j_{l+1})\big)-\frac{\delta}{2^k}=M_k.
	\end{array}
	\end{equation}
	
	\textbf{\underline{Case 3:}} Suppose that $j_k=2$ and $j_{k+1}=2$ (that is, we stay in the second board), then
	\begin{equation}
	\begin{array}{l}
		\displaystyle \E_{S_{I}^{\ast},S_{II}}^{(x_0,1)}[M_{k+1}|M_0,\dots ,M_k]
		\\[10pt]
	\displaystyle =\E_{S_{I}^{\ast},S_{II}}^{(x_0,1)}\Big[v^{\eps}(x_{k+1})-\eps^2\sum_{l=0}^{k}\big(h_1(x_l)\chi_{\{j=1\}}(j_{l+1})-h_2(x_l)\chi_{\{j=2\}}(j_{l+1})\big)-\frac{\delta}{2^{k+1}}|M_0,\dots ,M_k\Big]\\[10pt]
	\displaystyle \underbrace{=}_{j_{k+1}=2}\E_{S_{I}^{\ast},S_{II}}^{(x_0,1)}\Big[v^{\eps}(x_{k+1})+\eps^2h_2(x_k) \\[10pt]
	\qquad \quad \displaystyle -\eps^2\sum_{l=0}^{k-1}\big(h_1(x_l)\chi_{\{j=1\}}(j_{l+1})-h_2(x_l)\chi_{\{j=2\}}(j_{l+1})\big)-\frac{\delta}{2^{k+1}}|M_0,\dots ,M_k\Big].
	\end{array}
	\end{equation}
	Therefore, we have
	\begin{equation}
	\begin{array}{l}
		\displaystyle\E_{S_{I}^{\ast},S_{II}}^{(x_0,1)}[M_{k+1}|M_0,\dots ,M_k]
		\\[10pt]
	\displaystyle \geq \alpha_2\big(\half \sup_{y \in B_{\eps}(x_k)}v^{\eps}(y)-\frac{\delta}{2^{k+1}}+\half\inf_{y \in B_{\eps}(x_k)}v^{\eps}(y)\big)+(1-\alpha_2)\kint_{B_{\eps}(x_k)}v^{\eps}(y)dy+\eps^2h_2(x_k)\\[10pt]
	\displaystyle \qquad  -\eps^2\sum_{l=0}^{k-1}\big(h_1(x_l)\chi_{\{j=1\}}(j_{l+1})-h_2(x_l)\chi_{\{j=2\}}(j_{l+1})\big)-\frac{\delta}{2^{k+1}}\\[10pt]
	\displaystyle \geq J_2(v^{\eps})(x_k) -\eps^2\sum_{l=0}^{k-1}\big(h_1(x_l)\chi_{\{j=1\}}(j_{l+1})-h_2(x_l)\chi_{\{j=2\}}(j_{l+1})\big)-\frac{\delta}{2^k}\\[10pt]
	\displaystyle \geq\min\{J_1(u^{\eps})(x_k),J_2(v^{\eps})(x_k)\} -\eps^2\sum_{l=0}^{k-1}\big(h_1(x_l)\chi_{\{j=1\}}(j_{l+1})-h_2(x_l)\chi_{\{j=2\}}(j_{l+1})\big)-\frac{\delta}{2^k}\\[10pt]
	\displaystyle =v^{\eps}(x_k) -\eps^2\sum_{l=0}^{k-1}\big(h_1(x_l)\chi_{\{j=1\}}(j_{l+1})-h_2(x_l)\chi_{\{j=2\}}(j_{l+1})\big)-\frac{\delta}{2^k}=M_k.
	\end{array}
	\end{equation}
	\textbf{\underline{Case 4:}} Suppose that $j_k=2$ and $j_{k+1}=1$ (that is, we jump to the first board)
	\begin{equation}
	\begin{array}{ll}
	\displaystyle \E_{S_{I}^{\ast},S_{II}}^{(x_0,1)}[M_{k+1}|M_0,\dots ,M_k]\\[10pt]
	\displaystyle =\E_{S_{I}^{\ast},S_{II}}^{(x_0,1)}\Big[u^{\eps}(x_{k+1})-\eps^2\sum_{l=0}^{k}\big(h_1(x_l)\chi_{\{j=1\}}(j_{l+1})-h_2(x_l)\chi_{\{j=2\}}(j_{l+1})\big)-\frac{\delta}{2^{k+1}}|M_0,\dots ,M_k\Big]\\[10pt]
	\displaystyle \underbrace{=}_{j_{k+1}=1}\E_{S_{I}^{\ast},S_{II}}^{(x_0,1)}\Big[u^{\eps}(x_{k+1}) \\[10pt]
	\qquad \quad \displaystyle -\eps^2h_1(x_k)-\eps^2\sum_{l=0}^{k-1}\big(h_1(x_l)\chi_{\{j=1\}}(j_{l+1})-h_2(x_l)\chi_{\{j=2\}}(j_{l+1})\big)-\frac{\delta}{2^{k+1}}|M_0,\dots ,M_k\Big]
.
\end{array}
\end{equation}
	Hence,
	\begin{equation}
	\begin{array}{ll}
	\displaystyle \E_{S_{I}^{\ast},S_{II}}^{(x_0,1)}[M_{k+1}|M_0,\dots ,M_k]\\[10pt]
	\displaystyle
	\displaystyle \geq \alpha_1\big(\half \sup_{y \in B_{\eps}(x_k)}u^{\eps}(y)-\frac{\delta}{2^{k+1}}+\half\inf_{y \in B_{\eps}(x_k)}u^{\eps}(y)\big)+(1-\alpha_1)\kint_{B_{\eps}(x_k)}u^{\eps}(y)dy-\eps^2h_1(x_k)\\[10pt]
	\displaystyle \qquad -\eps^2\sum_{l=0}^{k-1}\big(h_1(x_l)\chi_{\{j=1\}}(j_{l+1})-h_2(x_l)\chi_{\{j=2\}}(j_{l+1})\big)-\frac{\delta}{2^{k+1}}\\[10pt]
	\displaystyle = J_1(u^{\eps})(x_k) -\eps^2\sum_{l=0}^{k-1}\big(h_1(x_l)\chi_{\{j=1\}}(j_{l+1})-h_2(x_l)\chi_{\{j=2\}}(j_{l+1})\big)-\frac{\delta}{2^k}\\[10pt]
	\displaystyle \geq\min\{J_1(u^{\eps})(x_k),J_2(v^{\eps})(x_k)\} -\eps^2\sum_{l=0}^{k-1}\big(h_1(x_l)\chi_{\{j=1\}}(j_{l+1})-h_2(x_l)\chi_{\{j=2\}}(j_{l+1})\big)-\frac{\delta}{2^k}\\[10pt]
	\displaystyle =v^{\eps}(x_k) -\eps^2\sum_{l=0}^{k-1}\big(h_1(x_l)\chi_{\{j=1\}}(j_{l+1})-h_2(x_l)\chi_{\{j=2\}}(j_{l+1})\big)-\frac{\delta}{2^k}=M_k.
	\end{array}
	\end{equation}
	
	Thus, gathering the four cases, we conclude that $M_k$ is a submartingale. 
	
	Using the \textit{OSTh} we obtain
	\begin{equation}
	\displaystyle \E_{S_{I}^{\ast},S_{II}}^{(x_0,1)}[M_{\tau\wedge k}]\geq M_0.
	\end{equation}
	Taking limit as $k\rightarrow\infty$ we get
	\begin{equation}
	\displaystyle \E_{S_{I}^{\ast},S_{II}}^{(x_0,1)}[M_{\tau}]\geq M_0.
	\end{equation}
	If we take $\inf_{S_{II}}$ and then $\sup_{S_{I}}$ we arrive to
	\begin{equation}
	\displaystyle \sup_{S_{I}}\inf_{S_{II}}\E_{S_{I},S_{II}}^{(x_0,1)}[M_{\tau}]\geq M_0.
	\end{equation}
	This inequality says that
	\begin{equation}
	\displaystyle \sup_{S_{I}}\inf_{S_{II}}\E_{S_{I},S_{II}}^{(x_0,1)}[\mbox{total payoff}]\geq u(x_0)-\delta.
	\end{equation}
	
   To prove an inequality in the opposite direction we fix a strategy for Player II as follows: Whenever $j_k=1$ Player II decides to stay in the second board if
   \begin{equation}
   \min\big\{J_1(u^{\eps})(x_k),J_2(v^{\eps})(x_k)\big\}=J_2(v^{\eps})(x_k)
   \end{equation} 
   and Player II decides to jump to the first board when
   \begin{equation}
   \min\big\{J_1(u^{\eps})(x_k),J_2(v^{\eps})(x_k)\big\}=J_1(v^{\eps})(x_k)
   \end{equation}
   If we play Tug-of-War (in both boards) the Player II chosses
   \begin{equation}
   x_{k+1}^{II}=S_{II}^{\ast}\big((x_0,j_0),\dots,(
   x_k,j_k)\big) \quad \mbox{such that} \inf_{y \in B_{\eps}(x_k)}w^{\eps}(y,j_{k+1})+\frac{\delta}{2^{k+1}}\geq w^{\eps}(x_{k+1}^{II},j_{k+1}).
   \end{equation}
   
   Given this strategy for Player II and any strategy for Player I, using similar computations like the ones we did before, we can prove that the sequence of random variables
   \begin{equation}
   N_k=w^{\eps}(x_k,j_k)-\eps^2\sum_{l=0}^{k-1}\big(h_1(x_l)\chi_{\{j=1\}}(j_{l+1})-h_2(x_l)\chi_{\{j=2\}}(j_{l+1})\big)+\frac{\delta}{2^k}
   \end{equation}
   is a supermartingale. Finally, using the \textit{OSTh} we arrive to
   \begin{equation}
   \inf_{S_{II}}\sup_{S_{I}}\E_{S_{I},S_{II}}^{(x_0,1)}[\mbox{total payoff}]\leq u^{\eps}(x_0)+\delta.
   \end{equation}
   
   Then, we have obtained
   \begin{equation}
   u^{\eps}(x_0)-\delta\leq \sup_{S_{I}}\inf_{S_{II}}\E_{S_{I},S_{II}}^{(x_0,1)}[\mbox{total payoff}]\leq \inf_{S_{II}}\sup_{S_{I}}\E_{S_{I},S_{II}}^{(x_0,1)}[\mbox{total payoff}]\leq u^{\eps}(x_0)+\delta
   \end{equation}
for any positive $\delta$. 

Analogously, we can prove that
\begin{equation}
	v^{\eps}(x_0)-\delta\leq \sup_{S_{I}}\inf_{S_{II}}\E_{S_{I},S_{II}}^{(x_0,2)}[\mbox{total payoff}]\leq \inf_{S_{II}}\sup_{S_{I}}\E_{S_{I},S_{II}}^{(x_0,2)}[\mbox{total payoff}]\leq v^{\eps}(x_0)+\delta.
\end{equation}   
   
Since $\delta$ is arbitrary, this proves that the game has a value,
$$
\sup_{S_{I}}\inf_{S_{II}}\E_{S_{I},S_{II}}^{(x_0,1)}[\mbox{total payoff}] = \inf_{S_{II}}\sup_{S_{I}}\E_{S_{I},S_{II}}^{(x_0,1)}[\mbox{total payoff}] = w(x_0,1)$$
and
$$
\sup_{S_{I}}\inf_{S_{II}}\E_{S_{I},S_{II}}^{(x_0,2)}[\mbox{total payoff}] = \inf_{S_{II}}\sup_{S_{I}}\E_{S_{I},S_{II}}^{(x_0,2)}[\mbox{total payoff}] = w(x_0,2)
$$
and that these functions coincide with the solution to the (DPP),
$$
w(x_0,1) = u^\eps (x_0) \qquad \mbox{and}  \qquad w(x_0,2) = v^\eps (x_0)
$$
as we wanted to show.
\end{proof}

Since solutions to the (DPP) coincide with the value of the game and this is unique,
we obtain uniqueness of solutions to the (DPP). 

\begin{corol}
There exists a unique solution to the (DPP), \eqref{DPP}.
\end{corol}

\begin{proof}
Existence follows from Theorem \ref{existencia.DPP}
	and uniqueness from the fact that in Theorem \ref{teo.5} we proved that
	any solution to the (DPP) coincides with the value function of the game, that is, it verifies 
	$$
		u^{\eps}(x)=\inf_{S_{II}}\sup_{S_{I}}\E_{S_{I},S_{II}}^{(x,1)}[\mbox{total payoff}]=\sup_{S_{I}}\inf_{S_{II}}\E_{S_{I},S_{II}}^{(x,1)}[\mbox{total payoff}]
	$$
	and 
	$$
	v^{\eps}(x)=\inf_{S_{II}}\sup_{S_{I}}\E_{S_{I},S_{II}}^{(x,2)}[\mbox{total payoff}]=\sup_{S_{I}}\inf_{S_{II}}\E_{S_{I},S_{II}}^{(x,2)}[\mbox{total payoff}].
	$$ 
\end{proof}


\subsection{Uniform convergence as $\eps \to 0$.} \label{sect-conv}


To obtain a convergent subsequence of the values of the game $u^\eps$ and $v^\eps$ we will use the following
Arzela-Ascoli type lemma. For its proof see Lemma~4.2 from \cite{MPRb}.

\begin{Lemma}\label{lem.ascoli.arzela} Let $$\{u^\eps : \overline{\Omega}
	\to \R\}_{\eps>0}$$ be a set of functions such that
	\begin{enumerate}
		\item there exists $C>0$ such that $|u^\eps (x)|<C$ for
		every $\eps >0$ and every $x \in \overline{\Omega}$,
		
		\bigskip
		
		\item \label{cond:2} given $\delta >0$ there are constants
		$r_0$ and $\eps_0$ such that for every $\eps < \eps_0$
		and any $x, y \in \overline{\Omega}$ with $|x - y | < r_0 $
		it holds
		$$
		|u^\eps (x) - u^\eps (y)| < \delta.
		$$
	\end{enumerate}
	Then, there exists  a uniformly continuous function $u:
	\overline{\Omega} \to \R$ and a subsequence still denoted by
	$\{u^\eps \}$ such that
	\[
	\begin{split}
	u^{\eps}\to u \qquad\textrm{ uniformly in }\overline{\Omega},
	\mbox{ as $\eps\to 0$.}
	\end{split}
	\]
\end{Lemma}

So our task now is to show that $u^\eps$ and $v^\eps$ both satisfy the hypotheses of the previous lemma. First, we
observe that we already proved that they are uniformly bounded (see Corollary \ref{corl-acot}). 

To obtain the second hypothesis of Lemma \ref{lem.ascoli.arzela} we will need to prove some technical lemmas. This part of the paper is delicate and involve the choice of 
particular strategies for the players. 

First of all, we need to find an upper bound for the expectation of the total number of plays, 
\begin{equation}
	\E[\tau].
\end{equation}
To this end we define an auxiliary game as follows: In the next lemma we play a Tug-of-War or random walk game in an annulus and one of the players uses the strategy 
of pointing to the center of the annulus when they play Tug-of-War. Then, no mater 
if we play Tug-of-War or random walk at each turn and no mater the strategy
used by the other player we can obtain a precise bound (in terms of the configuration of the annulus and the distance of the initial position to the inner boundary) for the expected number of plays
until one reaches the ball inside the annulus.

The key point here is that if one of the players pulls towards $0$ each time
that they play Tug-of-War then the expected number of plays is bounded 
above by a precise expression that scales as $\varepsilon^{-2}$ independently
of the game that is played at every round (Tug-of-War or random walk). This upper bound translates to our game (starting at any of the two boards) since the result
implies that if one of the players chooses 
 to pulls towards $0$ then, independently of the choice of the other player and
 independently of the board at which we play (that is, independently of the coin toss that selects the game that is played (Tug-of-War or random walk)), the game ends in an expected number of times that satisfies the obtained upper
 bound. See Remark \ref{2tau} below.

\begin{Lemma}
	\label{TEC1}
	Given $0<\delta<R$, let us consider the anular domain $B_R(0)\backslash B_{\delta}(0)$. In this domain we consider the following game: given $x\in B_R(0)\backslash B_{\delta}(0)$ the next position of the token can be chosen using the game Tug-of-War or a random walk. When Tug-of-War is played, one of the players pulls towards $0$. In all cases the next position is assumed to be in $B_{\eps}(x)\cap B_R(0)$. The game ends when the token reaches $B_{\delta}(0)$. Then, if $\tau^{\ast}$ is the exit time we have the estimate
	\begin{equation}
	\eps^2 \E^{x_0}[\tau^{\ast}]\leq C_1(R/\delta)\mbox{dist}(\partial B_{\delta}(0),x_0)+o(1)
	\end{equation}
	where $o(1)\rightarrow 0$ if $\eps\rightarrow 0$.
\end{Lemma} 

\begin{proof}
	Without loss of generality we can suppose that $S_{I}^{\ast}$ is to pull towards $0$. Let us call 
	\begin{equation}
	E_{\eps}(x)=\E_{S_{I}^{\ast},S_{II}}^x[\tau^{\ast}].
	\end{equation}
	Notice taht $E$ is radial and increasing in $r=|x|$. 
	Since our aim is to obtain a bound that is
	independent of the game (Tug-of-War or random walk) that is played at each round,
	if we try to maximize the expectation for the exit time, we have that the function $E$ 
	verifies
	\begin{equation}
	E_{\eps}(x)\leq \max\Big\{\big(\half\sup_{y \in B_{\eps}(x)\cap B_R(0)}E_{\eps}(y)+\half\inf_{y \in B_{\eps}(x)\cap B_R(0)}E_{\eps}(y),\kint_{B_{\eps}(x)\cap B_R(0)}E_{\eps}(y)dy\big)\Big\}+1.
	\end{equation}
	Hence, let us consider the DPP
		\begin{equation}
	\tilde{E}_{\eps}(x)=\max\Big\{\big(\half\sup_{y \in B_{\eps}(x)\cap B_R(0)}\tilde{E}_{\eps}(y)+\half\inf_{y \in B_{\eps}(x)\cap B_R(0)}\tilde{E}_{\eps}(y),\kint_{B_{\eps}(x)\cap B_R(0)}\tilde{E}_{\eps}(y)dy\big)\Big\}+1.
	\end{equation}
	Let us denote $F_{\eps}(x)=\eps^2 \tilde{E}_{\eps}(x)$, then we obtain
	\begin{equation}
	F_{\eps}(x)=\max\Big\{\big(\half\sup_{y \in B_{\eps}(x)\cap B_R(0)}F_{\eps}(y)+\half\inf_{y \in B_{\eps}(x)\cap B_R(0)}F_{\eps}(y),\kint_{B_{\eps}(x)\cap B_R(0)}F_{\eps}(y)dy\big)\Big\}+\eps^2.
	\end{equation}
	This induces us to look for a function $F$ such that 
	\begin{equation}
	\label{SisDPP}
	\left\lbrace
	\begin{array}{ll}
	\displaystyle F(x)\geq \kint_{B_{\eps}(x)}F(y)dy+\eps^2 \\[10pt]
	\displaystyle F(x)\geq \half\sup_{y \in B_{\eps}(x)}F(y)+\half\inf_{y \in B_{\eps}(x)}F(y)+\eps^2.
	\end{array}
	\right.
	\end{equation}
	We arrived to a sort of discrete version to the following inequalities 
	\begin{equation}
	\label{SisDif}
		\left\lbrace
	\begin{array}{ll}
	\displaystyle \Delta F(x)\leq -2(N+2), \quad & x\in B_{R+\eps}(0)\backslash \ol{B}_{\delta-\eps}(0) ,\\[10pt]
	\displaystyle \Delta^1_{\infty}F(x)\leq -2, & x\in B_{R+\eps}(0)\backslash \ol{B}_{\delta-\eps}(0).
	\end{array}
	\right.
	\end{equation}
	If we assume that $F$ is radial and increasing ir $r=|x|$ we get
	\begin{equation}
	\Delta^1_{\infty}F=\partial_{r r}F\leq \partial_{r r}F+\frac{N-1}{r}\partial_{r }F=\Delta F.
	\end{equation}
	Hence, to find a solution of \eqref{SisDif}. We can consider the problem
	\begin{equation}
	\label{EqPoisson}
	\left\lbrace
	\begin{array}{ll}
	\displaystyle \Delta F (x) =-2(N+2),\quad  & x\in B_{R+\eps}(0)\backslash \ol{B}_{\delta}(0),
	\\[8pt]
	\displaystyle F(x)=0, & x\in\partial B_{\delta}(0),\\[8pt]
	\displaystyle \frac{\partial F}{\partial \nu}(x) =0, & x\in\partial B_{R+\eps}(0).
	\end{array}
	\right.
	\end{equation}
	where $\frac{\partial F}{\partial \nu}$ refers to the outward 
	normal derivative. The solution to this problem takes the form
	\begin{equation}
	\begin{array}{ll}
	\displaystyle F(r)=-ar^2-br^{2-N}+c & \mbox{for } N>2,\\[8pt]
	\displaystyle F(r)=-ar^2-b\log(r)+c & \mbox{for } N=2. 
	\end{array}
	\end{equation}
with $a,b,c\in \R$ that depends of $\delta,R,\eps,N$. For example, for $N>2$, we obtain
that $a$, $b$ and $c$ are given by the solution to the following equations,
\begin{equation}
	\begin{array}{l}
	\displaystyle \Delta F=-2aN=-2(N+2),\\[8pt]
	\displaystyle \partial_{r }F(R+\eps)=-2a(R+\eps)-b(2-N)(R+\eps)^{1-N}=0,\\[8pt]
	\displaystyle F(\delta)=-a\delta^2-b\delta^{2-N}+c=0.
	\end{array}
\end{equation}
Observe that the resulting function $F(r)$ is increasing. 

	In this way we find $F$ that satisfies the inequalities \eqref{SisDif}. The classical calculation using Taylor's expansion shows that $F$ satisfies \eqref{SisDPP} for each $B_{\eps}(x)\subset B_{R+\eps}\backslash\ol{B}_{\delta-\eps}(0)$. Moreover, since $F$ is increasing in $r$, it holds that for each $x\in B_R(0)\backslash B_{\delta}(0)$,
	$$
	\displaystyle \kint_{B_{\eps}(x)\cap B_R(0)}F\leq\kint_{B_{\eps}(x)}F\leq F(x)-\eps^2
	$$ 
	and $$
	\displaystyle \half\sup_{y \in B_{\eps}(x)\cap B_R(0)}F+\half\inf_{y \in B_{\eps}(x)\cap B_R(0)}F\leq\half\sup_{y \in B_{\eps}(x)}F+\half\inf_{y \in B_{\eps}(x)}F\leq F(x)-\eps^2.
	$$	
	
	Consider the sequence of random variables $(M_k)_{k\geq 1}$ given by
	\begin{equation}
	M_k=F(x_k)+k\eps^2.
	\end{equation}
	Let us prove that $(M_k)_{k\geq 0}$ is a \textit{supermartingale}. Indeed, we have
	\begin{equation}
		\begin{array}{ll}
	\displaystyle	\E[M_{k+1}|M_0,\dots , M_k]=\E[F(x_{k+1})+(k+1)\eps^2|M_0,\dots , M_k]\\[10pt]
	\displaystyle \leq \max\big\{\half\sup_{y \in B_{\eps}(x)\cap B_R(0)}F(y)+\half\inf_{y \in B_{\eps}(x)\cap B_R(0)}F(y),\kint_{B_{\eps}(x)\cap B_R(0)}F(y)dy\big\}
	\\[10pt]
	\displaystyle
	\leq F(x_k)+k\eps^2
	\end{array}
	\end{equation}
if $x_k\in B_R(0)\backslash \ol{B_{\delta}(0)}$. Thus, $M_k$ is a supermartingale. Using the \textit{OSTh} we obtain
\begin{equation}
	\E[M_{\tau^{\ast}\wedge k}]\leq M_0.
\end{equation}
This means
\begin{equation}
	\E^{x_0}[F(x_{\tau^{\ast}\wedge k})+(\tau^{\ast}\wedge k)\eps^2]\leq F(x_0).
\end{equation}
Using that $x_{\tau^{\ast}}\in \ol{B_{\delta}(0)}\backslash\ol{B_{\delta-\eps}(0)}$ we get
\begin{equation}
	0\leq-\E^{x_0}[F(x_{\tau^{\ast}})]\leq o(1).
\end{equation}
Furthermore, the estimate
\begin{equation}
	0\leq F(x_0)\leq C(R/\delta)dist(\partial B_{\delta},x_0)
\end{equation}
holds for the solution of \eqref{EqPoisson}. Then, taking limit as $k\rightarrow\infty$ we obtain
\begin{equation}
	\eps^2\E[\tau^{\ast}]\leq F(x_0)-\E[F(x_{\tau^{\ast}})]\leq C(R/\delta)dist(\partial B_{\delta}(0),x_0)+o(1).
\end{equation}
This completes the proof.
\end{proof}

\begin{remark}
	\label{2tau} {\rm
	Suppose that we are playing the previous game in two boards with $\Om$ inside
	the annulus. Let us recall that $\Om$ satisfies an exterior sphere condition: there exists $\delta>0$ such that given $y\in\partial\Om$ there exists $z_y\in\R^N$ such that $B_{\delta}(z_y)\subset \R^N\backslash\Om$ and $ \overline{B_{\delta}(z_y)}\cap\overline{\Om}=\{ y \}$. Remark that if we have some $\delta_0$ that verifies the exterior sphere property for ball of that radious, then, the exterior sphere property is also verified for balls with radius $\delta$ for every $\delta<\delta_0$. Then, we can consider simultaneously the game defined in Lemma \ref{TEC1} in the annular domain $B_R(0)\backslash B_{\delta}(0)$ with $\Om\subset B_R(0)\backslash B_{\delta}(0)$, and with the same strategies for Player I and Player II, so that no mater which game, $J_1$ or $J_2$, is played in any of the two boards we obtain
	the bound for the expected number of plays
	given in Lemma \ref{TEC1}. That is, if in the two boards game we start for example in $(x_0,1)$ and Player I decides to stay in the first board and play $J_1$, in the one board game with the annular domain the third player decides to play Tug-of-War with $\alpha_1$ probability, or random walk with $1-\alpha_1$ probability, but if the player decides to jump to the second board and play $J_2$, then in the one board game the third player decides to play Tug-of-War with $\alpha_2$ probability and random walk with $1-\alpha_2$ probability. Thus, using that $\Om\subset B_R(0)\backslash B_{\delta}(0)$ we deduce that in the two boards game the exit time $\tau$ is smaller or equal than the exit time $\tau^{\ast}$ corresponding to the one board game considered in the previous lemma. This means that we have
\begin{equation}
\E[\tau]\leq\E[\tau^{\ast}].
\end{equation}
}
\end{remark}

Next we derive an estimate for the asymptotic uniform continuity of the so called non homogeneus $p$-Laplacian functions. 

\begin{Lemma}
	\label{TEC2}
	Let be $\Om$ as above, $h:\Om\rightarrow \R$ and $F:\R^N\backslash\Om\rightarrow \R$ two lipschitz functions. For $0<\beta<1$ let $\mu^{\eps}:\R^N\rightarrow \R$ be a function that verify the following DPP
	\begin{equation}
	\left\lbrace
	\begin{array}{ll}
	\displaystyle \mu^{\eps}(x)=\beta\big[\half\sup_{y \in B_{\eps}(x)}\mu^{\eps}(y)+\half\inf_{y \in B_{\eps}(x)}\mu^{\eps}(y)\big]+(1-\beta)\kint_{B_{\eps}(x)}\mu^{\eps}(y)dy+\eps^2 h(x), & x\in\Om, \\[10pt]
	\displaystyle \mu^{\eps}(x)=F(x), & x\in\R^N\backslash\Om.
	\end{array}
	\right.
	\end{equation}
	Then, given $\eta>0$ there exists $r_0>0$ and $\eps_0>0$ such that 
	\begin{equation}
	|\mu^{\eps}(x)-\mu^{\eps}(y)|<\eta
	\end{equation}
	if $|x-y|<r_0$ and $\eps<\eps_0$.
\end{Lemma}

\begin{proof}
	We have several cases:
	
	\textbf{Case 1:} If $x,y\in\R^N\backslash\Om$ we have
	\begin{equation}
	|\mu^{\eps}(x)-\mu^{\eps}(y)|=|F(x)-F(y)|\leq L(F)|x-y|<\eta
	\end{equation}
	if $r_0<\frac{\eta}{L(F)}$.
	
	\textbf{Case 2:} Suppose now that $x\in\Om$ and $y\in\partial\Om$. Without loss of generality we can suppose that $\Om\subset B_R(0)\backslash B_{\delta}(0)$ and $y\in\partial B_{\delta}(0)$. Let us call $x_0=x$ the first position in the game. In first case consider that the Player I use the strategy of pulling towards $0$, denoted by $S_{I}^{\ast}$. Let us consider the sequence of random variables
	\begin{equation}
	M_k=|x_k|-C\eps^2 k
	\end{equation}
	If $C>0$ is large enough $M_k$ is a supermartingale. Indeed
	 \begin{equation}
	 \begin{array}{ll}
	 \displaystyle \E_{S_{I}^{\ast},S_{II}}^{x_0}[|x_{k+1}| |x_0,\dots x_k]\leq \beta\big[\half(|x_k|+\eps)+\half(|x_k|-\eps)\big]+(1-\beta)\kint_{B_{\eps}(x_k)}|z|dz\\[10pt]
	 \displaystyle \leq |x_k|+C\eps^2.
	 \end{array}
	 \end{equation}
	 The first inequality follows form the choice of the strategy, and the second from the estimate
	 \begin{equation}
	 \kint_{B_{\eps}(x)}|z|dz\leq |x|+C\eps^2.
	 \end{equation}
	 Using the \textit{OSTh} we obtain
	 \begin{equation}
	 \E_{S_{I}^{\ast},S_{II}}^{x_0}[|x_{\tau}|]\leq |x_0|+C\eps^2\E_{S_{I}^{\ast},S_{II}}[\tau].
	 \end{equation}
	 Now, Lemma \ref{TEC1} and Remark \ref{2tau} give us the estimate
	 \begin{equation}
	 \eps^2 \E_{S_{I}^{\ast},S_{II}}^{x_0}[\tau]\leq\eps^2 \E_{S_{I}^{\ast},S_{II}}^{x_0}[\tau^{\ast}]\leq C_1(R/\delta)\mbox{dist}(\partial B_{\delta}(0),x_0)+o(1).
	 \end{equation}
	 Then,
	 \begin{equation}
	 \E_{S_{I}^{\ast},S_{II}}^{x_0}[|x_{\tau}|]\leq |x_0-y|+\delta+C_2(R/\delta)|x_0-y|+o(1).
	 \end{equation}
	 Here we call $C_2(R/\delta)= C_1(R/\delta)$. If we rewrite this inequality we obtain
	 \begin{equation}
	 \E_{S_{I}^{\ast},S_{II}}^{x_0}[|x_{\tau}|]\leq\delta+C_3(R/\delta)|x_0-y|+o(1)
	 \end{equation}
	 with $C_3(R/\delta)=C_2(R/\delta)+1$. 
	 
	 Using that $F$ is a Lipschitz function we have
	 \begin{equation}
	 |F(x_{\tau})-F(0)|\leq L(F)|x_{\tau}|.
	 \end{equation}
	 Hence, we get
	 \begin{equation}
	 \begin{array}{ll}
	 \displaystyle \E_{S_{I}^{\ast},S_{II}}^{x_0}[F(x_{\tau})]\geq F(0)-L(F)\E_{S_{I}^{\ast},S_{II}}^{x_0}[|x_{\tau}|]\\[10pt]
	 \displaystyle \geq F(y)-L(F)\delta-L(F)C_3(R/\delta)|x_0-y|+ o(1)
	 \\[10pt]
	 \displaystyle
	 \geq F(y)-L(F)\delta-L(F)C_3(R/\delta)r_0-o(1).
	 \end{array}
	 \end{equation}
	Then
	\begin{equation}
	\E_{S_{I}^{\ast},S_{II}}^{x_0}[F(x_{\tau})+\eps^2\sum_{j=0}^{\tau -1}h(x_j)]\geq F(y)-L(F)\delta-L(F)C_3 r_0 - \lVert h\rVert_{\infty} C r_0 - o(1).
	\end{equation}
	Thus, taking $\inf_{S_{II}}$ and then $\sup_{S_{I}}$ we get
	\begin{equation}
	\mu^{\eps}(x_0)>F(y)-L(F)\delta-L(F)C_3 r_0 - \lVert h\rVert_{\infty} C r_0 - o(1)>F(y)-\eta
	\end{equation}
 here we take $\delta>0$ such that $L(F)\delta<\frac{\eta}{3}$, then take $r_0>0$ such that $(L(F)C_3+\lVert h\rVert_{\infty}C)r_0<\frac{\eta}{3}$ and $o(1)<\frac{\eta}{3}$.
 
 Analogously, we can obtain the estimate 
 \begin{equation}
 \mu^{\eps}(x_0)<F(y)+\eta
 \end{equation}
  if the player II use the strategy that pull towards $0$. This ends the proof in this case.
  
 \textbf{Case 3:} Now, given two points $x$ and $y$ inside $\Omega$ with $|x-y|<r_0$ we couple the  
  game starting at $x_0=x$ with the game starting at $y_0=y$ making the same movements. This coupling generates two sequences of positions $x_i$ and $y_i$
  such that $|x_i - y_i|<r_0$ and $j_i=k_i$. This continues until one of the games exits the domain (say at $y_\tau \not\in \Omega$).
  At this point for the game starting at $x_0$ we have that its position $x_\tau$ is close to the exterior point $y_\tau \not\in \Omega$ (since we
  have $|x_\tau - y_\tau|<r_0$) and hence we can use our previous estimates for points close to the boundary to conclude that 
  $$ |\mu^{\eps}(x_{0})- \mu^\eps (y_0)|< \eta. $$
  
  This ends the proof.   
\end{proof}

Now we are ready to prove the second condition of the Arzela-Ascoli type result, Lemma \ref{lem.ascoli.arzela}. 

\begin{Lemma}
	Let $(u^{\eps},v^{\eps})$ be a pair of functions that is a solution to the (DPP) \eqref{DPP} given by
	\begin{equation}
	\left\lbrace
	\begin{array}{ll}
	\displaystyle u^{\eps}(x)=\max\Big\{ J_1(u^{\eps})(x), J_2(v^{\eps})(x)\Big\} ,
	\qquad &  x \in \Omega,  \\[10pt]
	\displaystyle v^{\eps}(x)=\min\Big\{ J_1(u^{\eps})(x), J_2(v^{\eps})(x)\Big\} ,
	&  x \in \Omega,  \\[10pt]
	u^{\eps}(x) = f(x), \qquad & x \in \R^{N} \backslash \Omega,  \\[10pt]
	v^{\eps}(x) = g(x), \qquad &  x \in \R^{N} \backslash \Omega.
	\end{array}
	\right.
	\end{equation}
	Given $\eta>0$, there exists $r_0>0$ and $\eps_0>0$ such that 
\begin{equation}
|u^{\eps}(x)-u^{\eps}(y)|<\eta \qquad \mbox{and} \qquad |v^{\eps}(x)-v^{\eps}(y)|<\eta
\end{equation}
if $|x-y|<r_0$ and $\eps<\eps_0$. 
\end{Lemma}

\begin{proof}
We will proceed by repeating the ideas used in Lemma \ref{TEC2}. 

We consider again several cases.

\textbf{Case 1:} Suppose that $x,y\in\R^N\backslash\Om$, then we have
$$
|u^{\eps}(x)-u^{\eps}(y)|=|f(x)-f(y)|\leq L(f)|x-y|<\eta $$
and 
$$
|v^{\eps}(x)-v^{\eps}(y)|=|g(x)-g(y)|\leq L(g)|x-y|<\eta
$$
if $\max\{L(f),L(g)\}r_0<\eta$.

	\textbf{Case 2:} Let us begin with the estimate of $u^{\eps}$. Suppose now that $x\in\Om$ and $y\in\partial\Om$ in the first board (we denote $(x,1)$ and $(y,1)$). Without loss of generality we suppose again that $\Om\subset B_R(0)\backslash B_{\delta}(0)$ and $y\in\partial B_{\delta}(0)$. Let us call $x_0=x$ the first position in the game. Player I uses the following strategy called $S_{I}^{\ast}$: the token always stay in the first board (Player I decides not to change boards), and pulls towards $0$ when Tug-of-War is played. 
	In this case we have that $u^{\eps}$ is a supersolution to the 
	DPP that appears in Lemma \ref{TEC2} (with $\beta=\alpha_1$). Notice that
	the game is always played in the first board.
	As Player I wants to maximize the expected value we get that the first
	component for our system, $u^\varepsilon$, verifies
	$$
	u^\varepsilon (x) \geq \mu^{\varepsilon} (x)
	$$
	(the value function when the player that wants to maximize
	is allowed to choose to change boards is bigger or equal than the value function 
	of a game where the player does not have the possibility of making this choice).
	From this bound and Lemma 6 a lower bound for  $u^\varepsilon$ close to the boundary follows. That is,
	from the estimate obtained in that lemma we get
	\begin{equation}
	 u^{\eps}(x)>f(y)-\eta.
	 \end{equation}
	Let us be more precise and consider the sequence of random variables 
\begin{equation}
M_k=|x_k|-C\eps^2 k.
\end{equation}
	We obtain arguing as before that $M_k$ is a supermartingale for $C>0$ large enough. If we repeat the reasoning of the Lemma \ref{TEC2} (this can be done because we stay in the first board) we arrive to
	 \begin{equation}
	 u^{\eps}(x)>f(y)-\eta
	 \end{equation}
if $|x-y|<r_0$ and $\eps<\eps_0$ for some $r_0$ and $\eps_0$. 

Now, the next estimate requires a particular strategy for Player II, $S_{II}^{\ast}$: when play the Tug-of-War game, Player II pulls towards $0$ (in both boards) and if in some step Player I decides to jump to the second board, then Player II decides to stay always in this board
and the position never comes back to the first board. Let us consider 
\begin{equation}
M_k=|x_k|-C\eps^2 k.
\end{equation}
We want to estimate 
\begin{equation}
\E_{S_{I},S_{II}^{\ast}}^{(x_0,1)}[|x_{k+1}| | x_0, \dots ,x_k].
\end{equation}
Now, Lemma \ref{TEC2} says that
\begin{equation}
\E_{S_{I},S_{II}^{\ast}}^{(x_0,1)}[|x_{k+1}| | x_0, \dots ,x_k]\leq |x_k|+C\eps^2
\end{equation}
for all posible combinations of $j_k$ and $j_{k+1}$. Using the \textit{OSTh} we obtain
\begin{equation}
\E_{S_{I},S_{II}^{\ast}}^{(x_0,1)}[|x_{\tau}|]\leq |x_0|+C\eps^2\E_{S_{I},S_{II}^{\ast}}[\tau].
\end{equation}
Let us suppose that $j_{\tau}=1$. This means that $j_k=1$ for all $0\leq k \leq \tau$. If we proceed as in Lemma \ref{TEC2}, we obtain 
\begin{equation}
\E_{S_{I},S_{II}^{\ast}}^{(x_0,1)}[\mbox{final payoff}]\leq f(y)+L\delta+LCr_0+\lVert h_1\rVert_{\infty}Cr_0+o(1).
\end{equation}
 On the other hand, if $j_{\tau}=2$, we have
 \begin{equation}
 \E_{S_{I},S_{II}^{\ast}}^{(x_0,1)}[g(x_{\tau})]\leq g(y)+L\delta+LC_3r_0+o(1)\leq f(y)+L\delta+LC_3r_0+o(1).
 \end{equation}
 Thus, we get
 \begin{equation}
 \begin{array}{l}
 \displaystyle
 \E_{S_{I},S_{II}^{\ast}}^{(x_0,1)}[g(x_{\tau})+\sum_{l=0}^{\tau -1}(h_1(x_l)\chi_{\{l=1\}}(l)+h_2(x_l)\chi_{\{l=2\}}(l))] \\[10pt]
 \displaystyle \leq f(y)+L\delta+LCr_0+(\lVert h_1\rVert +\lVert h_2\rVert_{\infty})Cr_0+o(1).
 \end{array}
 \end{equation}
 
 In both cases, taking $\sup_{S_{I}}$ and then $\inf_{S_{II}}$ we arrive to
 \begin{equation}
 u^{\eps}(x_0)\leq f(y)+\eta
 \end{equation}
 taking $\delta>0$, $r_0>0$ and $\eps>0$ small enough.
 
 Analogously we can obtain the estimates for $v^{\eps}$ and complete the proof.  
\end{proof}

As a corollary we obtain uniform convergence along a sequence $\eps_j \to 0$.

\begin{corol}
There exists a sequence $\eps_j \to 0$ and a pair of functions $(u,v)$ that are continuous in $\overline{\Omega}$ such that
$$
	u^{\varepsilon_j}\rightrightarrows u, \qquad v^{\varepsilon_j}\rightrightarrows v
	$$
	uniformly in $\overline{\Omega}$.
\end{corol}

\begin{proof}
The result follows from Lemma \ref{lem.ascoli.arzela}. 
\end{proof}


\subsection{The limit is a viscosity solution to the PDE system} \label{sect-lim-viscoso}

Our main goal in this section is to prove that the limit pair $(u,v)$ is a viscosity solution to 
\eqref{ED1}. 

First, let us state the precise definition of what we understand as a viscosity 
solution for the system \eqref{ED1}. We refer to \cite{CIL}
for a general reference to viscosity theory.

{\bf Viscosity solutions.}
We begin with the definition of a viscosity solution to a fully nonlinear second order elliptic PDE.
Fix a function
\[
P:\Omega\times\R\times\R^N\times\mathbb{S}^N\to\R
\]
where $\mathbb{S}^N$ denotes the set of symmetric $N\times N$ matrices and
consider the PDE 
\begin{equation}
\label{eqvissol}
P(x,u (x), Du (x), D^2u (x)) =0, \qquad x \in \Omega.
\end{equation}
We will assume that $P$ is degenerate elliptic, that is, $P$ satisfies a monotonicity property with respect
to the matrix variable, that is,
\[
X\leq Y \text{ in } \mathbb{S}^N \implies P(x,r,p,X)\geq P(x,r,p,Y)
\]
for all $(x,r,p)\in \Omega\times\R\times\R^N$.

\begin{definition}
	\label{def.sol.viscosa.1}
	A lower semi-continuous function $ u $ is a viscosity
	supersolution of \eqref{eqvissol} if for every $ \phi \in C^2$ such that $ \phi $
	touches $u$ at $x \in \Omega$ strictly from below (that is, $u-\phi$ has a strict minimum at $x$ with $u(x) = \phi(x)$), we have
	$$P(x,\phi(x),D \phi(x),D^2\phi(x))\geq 0.$$
	
	An upper semi-continuous function $u$ is a subsolution of \eqref{eqvissol} if
	for every $ \psi \in C^2$ such that $\psi$ touches $u$ at $ x \in
	\Omega$ strictly from above (that is, $u-\psi$ has a strict maximum at $x$ with $u(x) = \psi(x)$), we have
	$$P(x,\phi(x),D \phi(x),D^2\phi(x))\leq 0.$$
	
	Finally, $u$ is a viscosity solution of \eqref{eqvissol} if it is both a super- and a subsolution.
	
	When $P$ is not continuous one has to consider the upper and lower semicontinuous envelopes of $P$, that we denote by $P^*$ and $P_*$ respectively, and consider 
	$$P^*(x,\phi(x),D \phi(x),D^2\phi(x))\geq 0,$$
	and 
	$$P_*(x,\phi(x),D \phi(x),D^2\phi(x))\leq 0,$$
	when defining super- and subsolutions.
\end{definition}

In our system \eqref{ED1} we have two equations given by the functions
$$
F_1 (x,u,v,p,X)  =  \min\Big\{-  \frac{\alpha_1}{2} \langle X \frac{p}{|p|}, \frac{p}{|p|} \rangle  - \frac{(1-\alpha_1)}{2(N+2)} trace(X)+h_1(x),(u-v)(x)\Big\}  
$$
and 
$$
F_2  (x,u,v,q,Y) =\max \Big\{ -  \frac{\alpha_2}{2} \langle Y \frac{q}{|q|}, \frac{q}{|q|}\rangle  - \frac{(1-\alpha_2)}{2(N+2)}trace(Y)-h_2(x),(v-u)(x)\Big\}.
$$

These functions $F_1$ and $F_2$ are not continuous (they are not even well defined for $p=0$ and for $q=0$ respectively). The upper semicontinuous envelope of $F_1$ is given by  
$$
(F_1)^* (x,u,v,p,X)  = 
\left\{
\begin{array}{ll}
\displaystyle  \min\Big\{-  \frac{\alpha_1}{2} \langle X \frac{p}{|p|}, \frac{p}{|p|} \rangle  - \frac{(1-\alpha_1)}{2(N+2)} trace(X)+h_1(x),(u-v)(x)\Big\} , \quad p\neq 0, \\[10pt]
 \displaystyle  \min\Big\{-  \frac{\alpha_1}{2} \lambda_1 (X)   - \frac{(1-\alpha_1)}{2(N+2)} trace(X)+h_1(x),(u-v)(x)\Big\} , \quad p = 0.
 \end{array} \right.
$$
Here $\lambda_1 (X) = \min \{ \lambda : \lambda \mbox{ is an eigenvalue of } X \}$.
While the lower semicontinuous envelope is 
$$
(F_1)_* (x,u,v,p,X)  = 
\left\{
\begin{array}{ll}
\displaystyle  \min\Big\{-  \frac{\alpha_1}{2} \langle X \frac{p}{|p|}, \frac{p}{|p|} \rangle  - \frac{(1-\alpha_1)}{2(N+2)} trace(X)+h_1(x),(u-v)(x)\Big\} , \quad p\neq 0, \\[10pt]
 \displaystyle  \min\Big\{-  \frac{\alpha_1}{2} \lambda_N (X)   - \frac{(1-\alpha_1)}{2(N+2)} trace(X)+h_1(x),(u-v)(x)\Big\} , \quad p= 0.
 \end{array} \right.
$$
Here $\lambda_N (X) = \max \{ \lambda : \lambda \mbox{ is an eigenvalue of } X \}$.

Analogous formulas hold for $(F_2)^*$ and $(F_2)_*$ changing $\alpha_1$ by $\alpha_2$.

Then, the definition of a viscosity solution for the system \eqref{ED1} that we will use here is the following.

\begin{definition}
	\label{def.sol.viscosa.system}
	A pair of continuous functions $ u, v :\overline{\Omega} \mapsto \mathbb{R} $ is a viscosity
	solution to \eqref{ED1} if
	\begin{enumerate}
\item	$$
	u(x) \geq v(x) \qquad x \in \overline{\Omega},
	$$
	\item
	$$
	u|_{\partial \Omega} = f, \qquad v|_{\partial \Omega} = g,
	$$
	\item
	$$
	\mbox{$u$ is a viscosity solution to 
		$F_1 (x,u,v(x),\nabla u, D^2u) = 0$ in $\{x:u(x) > v(x)\}$ }
	$$
	$$
	\mbox{$u$ is a viscosity supersolution to 
		$F_1 (x,u,v(x),\nabla  u, D^2u) = 0$ in $\Omega$ ,}
	$$
	\item
	$$
	\mbox{$v$ is a viscosity solution to 
		$
		F_2 (x,u(x),v,\nabla v, D^2v) = 0$ in $\{x:u(x) > v(x)\}$}
	$$
	$$
	\mbox{$v$ is a viscosity solution to 
		$
		F_2 (x,u(x),v,\nabla v, D^2v) = 0$ in $\Omega$}.
	$$
	\end{enumerate}
\end{definition}

\begin{remark} {\rm The meaning of Definition \ref{def.sol.viscosa.system}
is that
		we understand a solution to \eqref{ED1} as a pair of continuous up to the boundary functions
		that satisfies the boundary conditions pointwise and such that $u$ is a viscosity solution to 
		the obstacle problem (from below) for the first equation in the system (with $v$ as a fixed continuos function of $x$ as obstacle from below) and $v$ solves the obstacle problem (from above) for the second equation 
		in the system (regarding $u$ as a fixed function of $x$ as obstacle from above). 
	}
\end{remark}

With this definition at hand we are ready to show that any uniform limit of the value functions
of our game is a viscosity solution to the two membranes problem with the two different $p-$Laplacians.

\begin{theorem}  \label{teo.sol.viscosa} Let $(u,v)$ be continuous functions that are a uniform
	limit of a sequence of values of the game, that is,
	$$
	u^{\varepsilon_j}\rightrightarrows u, \qquad v^{\varepsilon_j}\rightrightarrows v
	$$
	uniformly in $\overline{\Omega}$ as $\eps_j \to 0$. 
	Then, the limit pair $(u,v)$ is a viscosity solution to 
	\eqref{ED1}. 
\end{theorem} 

\begin{proof} We divide the proof in several cases.

{\bf 1) $u$ and $v$ are ordered.} From the fact that 
$$
u^{\varepsilon_j} \geq v^{\varepsilon_j}
$$
in $\overline{\Omega}$ and the uniform convergence we immediately get
$$
u \geq v
$$
in $\overline{\Omega}$.

{\bf 2) The boundary conditions.} As we have that
$$
u^{\varepsilon_j} = f, \qquad  v^{\varepsilon_j} = g,
$$
in $\mathbb{R}^N \setminus \Omega$ we get
$$
	u|_{\partial \Omega} = f, \qquad v|_{\partial \Omega} = g,
$$
	
	\textbf{3) The equation for $u$.} First, let us show that $u$ is a viscosity supersolution to 
	$$
	-\Delta_p^1 u(x)+h_1(x)= 0
	$$ 
	for $x\in \Omega$. To this end, consider a point $x_0\in \Omega$ and a smooth function
	$\varphi\in C^2(\Omega)$ such that $(u-\varphi)(x_0)=0$ is a strict minimum of $(u-\varphi)$. Then, from the uniform convergence there exists sequence of points, that we will denote by $(x_{\eps})_{\eps>0}$, such that $x_{\eps}\rightarrow x_0$ and 
	\begin{equation}
	(u^{\eps}-\varphi)(x_{\eps})\leq(u^{\eps}-\varphi)(y)+o(\eps^2),
	\end{equation}
	that is, 
	\begin{equation}
	\label{ump}
	u^{\eps}(y)-u^{\eps}(x_{\eps})\geq\varphi(y)-\varphi(x_{\eps})-o(\eps^2).
	\end{equation}
	From the DPP \eqref{DPP} we have 
	\begin{equation}
	\label{desDPP}
	0=\max\Big\{ J_1(u^{\eps})(x_{\eps})-u(x_{\eps}), J_2(v^{\eps})(x_{\eps})-u^{\eps}(x_{\eps}) \Big\}\geq J_1(u^{\eps})(x_{\eps})-u^{\eps}(x_{\eps}),
	\end{equation}
	Writting $J_1(u^{\eps})(x_{\eps})-u(x_{\eps})$ we obtain
	\begin{equation}
	\label{J1menosu}
	\begin{array}{ll}
	\displaystyle J_1(u^{\eps})(x_{\eps})-u^{\eps}(x_{\eps}) \\[8pt]
	\displaystyle = \alpha_1\Big[\half \sup_{y \in B_{\eps}(x_{\eps})}(u^{\eps}(y)-u^{\eps}(x_{\eps})) + \half \inf_{y \in B_{\eps}(x_{\eps})}(u^{\eps}(y)-u^{\eps}(x_{\eps}))\Big]
	\\[10pt]
	\qquad \displaystyle +(1-\alpha_1)\kint_{B_{\eps}(x_{\eps})}(u^{\eps}(y)-u^{\eps}(x_{\eps}))dy-\eps^2h_1(x_{\eps}),
	\end{array}
	\end{equation}
	and then, using \eqref{ump}, we get
		\begin{equation}
	\label{J1conphi}
	\begin{array}{ll}
	\displaystyle J_1(u^{\eps})(x_{\eps})-u^{\eps}(x_{\eps}) \\[8pt]
	\displaystyle \geq \alpha_1\underbrace{\Big[\half \sup_{y \in B_{\eps}(x_{\eps})}(\varphi(y)-\varphi(x_{\eps})) + \half \inf_{y \in B_{\eps}(x_{\eps})}(\varphi(y)-\varphi(x_{\eps}))\Big]}_{I} \\[8pt]
	\qquad \displaystyle
	+\beta_1\underbrace{\kint_{B_{\eps}(x_{\eps})}(\varphi(y)-\varphi(x_{\eps}))dy}_{II}-\eps^2h_1(x_{\eps})+o(\eps^2).
	\end{array}
	\end{equation}
	Let us analyze $I$ and $II$. We begin with $I$: 
	Assume that $\nabla\varphi(x_0)\neq 0$. Let $z_{\eps}\in B_1(0)$ be such that
	$$
	\max_{y \in \ol{B}_{\eps}(x_{\eps})}\varphi(y)=\varphi(x_{\eps}+\eps z_{\eps}).
	$$
	Then, we have
	\begin{equation}
	I=\half(\varphi(x_{\eps}+\eps z_{\eps})-\varphi(x_{\eps}))+\half(\varphi(x_{\eps}-\eps z_{\eps})-\varphi(x_{\eps}))+o(\eps^2).
	\end{equation}
	From a simple Taylor expansion we conclude that
	\begin{equation}
	\half(\varphi(x_{\eps}+\eps z_{\eps})-\varphi(x_{\eps}))+\half(\varphi(x_{\eps}-\eps z_{\eps})-\varphi(x_{\eps}))+o(\eps^2)=\half\eps^2\langle D^2\varphi(x_{\eps})z_{\eps},z_{\eps}\rangle +o({\eps^2}).
	\end{equation}
	Dividing by $\eps^2$ the first inequality and taking the limit as $\eps\rightarrow 0$ (see \cite{nosotros})
	we arrive to  
	\begin{equation}
	\label{limI}
	I\rightarrow \half\Delta^1_{\infty}\varphi(x_0)
	\end{equation}
	
	When $\nabla\varphi=0$ arguing again using Taylor's expansions we get
	$$
	\limsup I \geq \half \lambda_1 (D^2 \varphi (x_0))
	$$
	see \cite{BRLibro} for more details.

	Now, we look at $II$: Using again Taylor's expansions we obtain
	$$ \kint_{B_{\eps}(x_{\eps})}(\varphi(y)-\varphi(x_{\eps}))dy=\frac{\eps^2}{2(N+1)}\Delta \varphi(x_{\eps})+o(\eps^2),$$
	Dividing by $\eps^2$ and taking limits as $\eps \rightarrow 0$ we get
	\begin{equation}
	\label{limII}
	\displaystyle II\rightarrow \frac{1}{2(N+2)}\Delta \varphi(x_{0}).
	\end{equation}
	
	Therefore, if we come back to \eqref{desDPP}, dividing by $\eps^2$ and taking limit $\eps\rightarrow 0$ we obtain
	\begin{equation}
	0\geq \frac{\alpha_1}{2}\Delta^1_{\infty}\varphi(x_0)+\frac{1-\alpha_1}{2(N+1)}\Delta\varphi(x_0)-h_1(x_0)
	\end{equation} 
	when $\nabla \varphi (x_0) \neq 0$ and 
	\begin{equation}
	0\geq \frac{\alpha_1}{2}\lambda_1(D^2\varphi(x_0))+\frac{1-\alpha_1}{2(N+1)}\Delta\varphi(x_0)-h_1(x_0)
	\end{equation} 
	when $\nabla \varphi (x_0) = 0$.
	
	Using the definition of the normalized $p-$Laplacian we have arrived to 
	\begin{equation}
	-\Delta_p^1\varphi(x_0)+h_1(x_0)\geq 0
	\end{equation}
	in the sense of Definition \ref{def.sol.viscosa.1}.

	Now we are going to prove that $u$ is viscosity solution to
	\begin{equation}
	\label{parmonica}
	-\Delta_p^1 u(x)+h_1(x)= 0
	\end{equation}
	in the set $\Om\cap\{ u>v \}$. Let us consider $x_0\in\Om\cap\{ u>v \}$. Let $\eta>0$ be such that 
	\begin{equation}
	u(x_0)\geq v(x_0)+3\eta. 
	\end{equation}
	Then, using that $u$ and $v$ are continuous functions, there exists $\delta>0$ such that
	\begin{equation}
	u(y)\geq v(y)+2\eta \quad \mbox{for all} \quad y\in B_{\delta}(x_0),
	\end{equation}
	and, using that $u^{\eps}\rightrightarrows u$ and $v^{\eps}\rightrightarrows v$ we have
	\begin{equation}
	u^{\eps}(y)\geq v^{\eps}(y)+\eta \quad \mbox{for all} \quad y\in B_{\delta}(x_0)  
	\end{equation}
	 for $0<\eps<\eps_0$ for some $\eps_0>0$. Given $z\in B_{\frac{\delta}{2}}(x_0)$ and $\eps<\min\{\eps_0,\frac{\delta}{2}\}$ we obtain 
	 	\begin{equation}
	 	B_{\eps}(z)\subset B_{\delta}(x_0).
	 	\end{equation}
	Using that $u^{\eps}\rightrightarrows u$ we obtain the following limits:
	\begin{equation} \label{Claim1}
	\displaystyle \sup_{y \in B_{\eps}(z)}u^{\eps}(y)\rightarrow u(z), \qquad \mbox{ as } \eps\rightarrow 0 
	\end{equation}
	In fact, from our previous estimates we have that 
	\begin{equation}
	\Big|\sup_{y \in B_{\eps}(z)}u^{\eps}(y)-u(z) \Big|\leq \sup_{y \in B_{\eps}(z)}|u^{\eps}(y)-u(y)|+
	\sup_{y \in B_{\eps}(z)}|u(y)-u(z)|
	\end{equation}
	Using that $u^{\eps}\rightrightarrows u$, exists $\eps_1>0$ such that if $\eps<\eps_1$
	\begin{equation}
	|(u^{\eps}-u)(x)|<\frac{\theta}{2} \quad \mbox{for all} \quad x\in\Om.
	\end{equation}
	Now, using that $u$ is continuous, exists $\eps_2>0$ such that
	\begin{equation}
	|u(y)-u(z)|<\frac{\theta}{2} \quad \mbox{if} \quad |y-z|<\eps_2,
	\end{equation}
	thus, if we take $\eps<\min\{\eps_0,\eps_1,\eps_2,\frac{\delta}{2} \}$ we obtain
	\begin{equation}
	\Big|\sup_{y \in B_{\eps}(z)}u^{\eps}(y)-u(z)\Big|<\theta.
	\end{equation}
	This proves \eqref{Claim1}.
	
	Also, with a similar argument, we have, 
		\begin{equation} \label{Claim2}
	\displaystyle \inf_{y \in B_{\eps}(z)}u^{\eps}(y)\rightarrow_{\eps\rightarrow 0} u(z).
	\end{equation}
	
	Finally, we get,
		\begin{equation} \label{Claim3}
	\displaystyle \kint_{B_{\eps}(z)}u^{\eps}(y)dy\rightarrow_{\eps\rightarrow 0} u(z).
	\end{equation}
	In fact, let us compute
	\begin{equation}
	\left|\kint_{B_{\eps}(z)}u^{\eps}(y)dy-u(z)\right|
	\leq \kint_{B_{\eps}(z)} \left| u^{\eps}(y)-u(y) \right| dy+\kint_{B_{\eps}(z)} \left| 
	u(y)-u(z)\right| dz.
	\end{equation}
	Now we use again that $u^{\eps}\rightrightarrows u$ and that $u$ is a continuous function to obtain 
	$$\kint_{B_{\eps}(z)}|u^{\eps}(y)-u(y)|dy<\frac{\theta}{2} \qquad \mbox{ and }\qquad 
	\kint_{B_{\eps}(z)}|u(y)-u(z)|dz<\frac{\theta}{2}$$ for $\eps>0$ small enough. Thus we obtain
	\begin{equation}
	\left|\kint_{B_{\eps}(z)}u^{\eps}(y)dy-u(z) \right|<\theta.
	\end{equation}

	Using the previous limits, \eqref{Claim1}, \eqref{Claim2} and \eqref{Claim3} we obtain
	\begin{equation} 
	J_1(u^{\eps})(z)\rightarrow u(z) \qquad \mbox{ as } \eps\rightarrow 0. 
	\end{equation}
	
	Analogously, we can prove that
	\begin{equation}
	J_2(v^{\eps})(z)\rightarrow v(z), \qquad \mbox{ as } \eps\rightarrow 0.
	\end{equation}
	
	Now, if we recall that $u(z)\geq v(z)+2\eta$, we obtain
	\begin{equation} \label{river-plate}
	J_1(u^{\eps})(z)\geq J_2(v^{\eps})(z)+\eta,
	\end{equation}
	if $\eps>0$ is small enough. Then, using de DPP we obtain
	\begin{equation}
	\label{DPPu}
	u^{\eps}(z)=\max\{J_1(u^{\eps})(z),J_2(v^{\eps})(z)\}=J_1(u^{\eps})(z),
	\end{equation}
	for all $z\in B_{\frac{\delta}{2}}(x_0)$ and for every $\eps>0$ small enough. 
	Let us prove that $u$ is  viscosity subsolution of the equation \eqref{parmonica}. Given now $\varphi\in\mathscr{C}^2(\Om)$ such that $(u-\varphi)(x_0)=0$ is maximum of $u-\varphi$. Then, from the uniform convergence there exists sequence of points $(x_{\eps})_{\eps>0}\subset B_{\frac{\delta}{2}}(x_0)$, such that $x_{\eps}\rightarrow x_0$ and 
	\begin{equation}
	(u^{\eps}-\varphi)(x_{\eps})\geq(u^{\eps}-\varphi)(y)+o(\eps^2),
	\end{equation}
	that is, 
	\begin{equation}
	\label{ump2}
	u^{\eps}(y)-u^{\eps}(x_{\eps})\leq\varphi(y)-\varphi(x_{\eps})-o(\eps^2).
	\end{equation}
		From the DPP \eqref{DPP} we have 
	\begin{equation}
	\label{igualDPP}
	0=\max\Big\{ J_1(u^{\eps})(x_{\eps})-u(x_{\eps}), J_2(v^{\eps})(x_{\eps})-u^{\eps}(x_{\eps}) \Big\}= J_1(u^{\eps})(x_{\eps})-u^{\eps}(x_{\eps}),
	\end{equation}
	Writting $J_1(u^{\eps})(x_{\eps})-u(x_{\eps})$ we obtain
	\begin{equation}
	\label{J1menosuigual}
	\begin{array}{ll}
	\displaystyle J_1(u^{\eps})(x_{\eps})-u^{\eps}(x_{\eps}) \\[8pt]
	\displaystyle = \alpha_1\Big[\half \sup_{y \in B_{\eps}(x_{\eps})}(u^{\eps}(y)-u^{\eps}(x_{\eps})) + \half \inf_{y \in B_{\eps}(x_{\eps})}(u^{\eps}(y)-u^{\eps}(x_{\eps}))\Big]
	\\[10pt]
	\displaystyle \qquad
	+(1-\alpha_1)\kint_{B_{\eps}(x_{\eps})}(u^{\eps}(y)-u^{\eps}(x_{\eps}))dy-\eps^2h_1(x_{\eps}),
	\end{array}
	\end{equation}
		and then, using \eqref{ump2}, we get
	\begin{equation}
	\label{J1conphi2}
	\begin{array}{ll}
	\displaystyle J_1(u^{\eps})(x_{\eps})-u^{\eps}(x_{\eps}) \\[8pt]
	\displaystyle \leq \alpha_1\Big[\half \sup_{y \in B_{\eps}(x_{\eps})}(\varphi(y)-\varphi(x_{\eps})) + \half \inf_{y \in B_{\eps}(x_{\eps})}(\varphi(y)-\varphi(x_{\eps}))\Big] \\[8pt]
	\qquad \displaystyle +(1-\alpha_1)\kint_{B_{\eps}(x_{\eps})}(\varphi(y)-\varphi(x_{\eps}))dy-\eps^2h_1(x_{\eps})+o(\eps^2).
	\end{array}
	\end{equation}
	Passing to the limit as before we obtain 
	\begin{equation}
	0\leq \frac{\alpha_1}{2}\Delta^1_{\infty}\varphi(x_0)+\frac{(1-\alpha_1)}{2(N+1)}\Delta\varphi(x_0)-h_1(x_0)
	\end{equation} 
	when $\nabla \varphi (x_0) \neq 0$ and 
	\begin{equation}
	0\leq \frac{\alpha_1}{2}\lambda_N (D^2 \varphi(x_0))+\frac{(1-\alpha_1)}{2(N+1)}\Delta\varphi(x_0)-h_1(x_0)
	\end{equation}
	if $\nabla \varphi (x_0) = 0$.
	Hence we arrived to 
	\begin{equation}
	-\Delta_p^1\varphi(x_0)+h_1(x_0)\leq 0,
	\end{equation}
	according to Definition \ref{def.sol.viscosa.1}. 
	This proves that $u$ is viscosity subsolution of the equation \eqref{parmonica}
	inside the open set $\{u>v\}$.
	
	 As we have that $u$ is a viscosity supersolution in the whole $\Omega$, we conclude that 
	 $u$ is viscosity solution to
	\begin{equation}
	-\Delta^1_p u(x_0)+h_1(x_0)= 0,
	\end{equation} 
    in the set $\{u>v\}$.
	
	\textbf{4) The equation for $v$.} The case that $v$ is a viscosity subsolution to 
	$$
	-\Delta^1_q v(x)+h_2(x)= 0
	$$ 
	is analogous. Here we use that 
	\begin{equation}
	\label{desDPP-v}
	0=\min\Big\{ J_2(v^{\eps})(x_{\eps})-v(x_{\eps}), J_1(u^{\eps})(x_{\eps})-v^{\eps}(x_{\eps}) \Big\}\leq J_2(v^{\eps})(x_{\eps})-v^{\eps}(x_{\eps}).
	\end{equation}
	
	To show that $v$ is a viscosity solution 
	\begin{equation}
	-\Delta_q^1 v(x_0)-h_2(x_0)= 0,
	\end{equation}
	if $x_0\in\Om\cap\{u>v\}$ we proceed as before. 
\end{proof}

\section{A game that gives an extra condition on the contact set}
\label{sect-second-game}

in this section we will study the value functions of the second game. 
In this case, they are given by a
 pair of functions $(u^{\eps},v^{\eps})$ that verifies the (DPP)
  \begin{equation}
 \label{DPPEXT2}
 \left\lbrace
 \begin{array}{ll}
 \displaystyle u^{\eps}(x)=\half\max\Big\{ J_1(u^{\eps})(x), J_2(v^{\eps})(x)\Big\}+\half J_1(u^{\eps})(x) 
 &  x \in \Omega,  \\[10pt]
 \displaystyle v^{\eps}(x)=\half\min\Big\{ J_1(u^{\eps})(x), J_2(v^{\eps})(x)\Big\} +\half J_2(v^{\eps})(x)
 &  x \in \Omega,  \\[10pt]
 u^{\eps}(x) = f(x) \qquad & x \in \R^{N} \backslash \Omega,  \\[10pt]
 v^{\eps}(x) = g(x) \qquad &  x \in \R^{N} \backslash \Omega.
 \end{array}
 \right.
 \end{equation}
Is clear from the (DPP) that
\begin{equation}
	u^{\eps}\ge v^{\eps}.
\end{equation}
We aim to show that these functions converge (along subsequences $\eps_j \to 0$) to a pair of functions $(u,v)$ that is a viscosity solution to the system
 \begin{equation}
 \label{EDEX2}
\left\lbrace
\begin{array}{ll}
\displaystyle u (x) \geq v(x), \qquad & \ \Omega, \\[10pt]
\displaystyle -\Delta_{p}^{1}u(x)+ h_1(x)\geq 0, \quad  \quad -\Delta_q^1 v(x)- h_2(x)\leq 0 ,\qquad & \ \Om,\\[10pt] 
\displaystyle -\Delta_p^1 u(x)+ h_1(x)=0, \quad  \quad -\Delta_q^1 v(x)- h_2(x)=0,\qquad & \ \{u>v\}\cap\Om,\\[10pt]
\displaystyle \big(-\Delta_{p}^{1}u(x)+ h_1(x)\big)+\big(-\Delta_q^1 v(x)- h_2(x)\big) =0, & \ x\in \Om,\\[10pt]
u(x) = f(x), \qquad & \ x \in \partial \Omega,  \\[10pt]
v(x) = g(x), \qquad & \ x \in \partial \Omega.
\end{array}
\right.
\end{equation}
Notice that here we have the classical formulation of the two membranes problem, 
but with the extra condition 
$$
\big(-\Delta_{p}^{1}u(x)+ h_1(x)\big)+\big(-\Delta_q^1 v(x)- h_2(x)\big) =0 
$$
that is meaningful for $x \in \{u(x) = v(x) \}$.

The existence and uniqueness of the pair of functions $(u^{\eps},v^{\eps})$ can be proved as before. In fact, we can reproduce the arguments of Perron's method to obtain existence
of a solution. Next, we show that given a solution to the (DPP) we can build quasioptimal
strategies and show that the game has a value and that this value coincides with 
the solution to the (DPP), from where uniqueness of solutions to the (DPP) follows.

Uniform convergence also follows with the same arguments used before using the Arzela-Ascoli type result together with the estimates close to the boundary proved in the previous section. 
Notice that here we can prescribe the same strategies as the ones used before. For example, Player I may decide to stay in the first board (if the coin toss allows a choice) and to point to a prescribed point when Tug-of-War is played. Also remark that the crucial bound on the expected number of plays given in Lemma \ref{TEC1}
 can also be used here to obtain a bound for the total number of plays in the variant of the game.

Passing to the limit in the viscosity sense is also analogous. One only has to pay special attention to the extra condition. Therefore,
let us prove now that the extra condition in \eqref{EDEX2},
\begin{equation}
\displaystyle \big(-\Delta_{p}^{1}u(x)+ h_1(x)\big)+\big(-\Delta_q^1 v(x)- h_2(x)\big) =0 
\end{equation}
holds in the viscosity sense, in the set $\{x: u(x)=v(x)\}$. 

Let us start proving the subsolution case. 
Given $x_0\in\{u=v\}$ and $\varphi\in\mathscr{C}^2(\Om)$ such that $(u-\varphi)(x_0)=0$ is maximum of $u-\varphi$. 
Notice that since $v(x_0)=u(x_0)$ and $v\leq u$ in $\Omega$ we also have that 
$(v-\varphi)(x_0)=0$ is maximum of $v-\varphi$. 
Then, from the uniform convergence there exists sequence of points $(x_{\eps})_{\eps>0}\subset B_{\frac{\delta}{2}}(x_0)$, such that $x_{\eps}\rightarrow x_0$ and 
\begin{equation}
	\label{umenosphi}
(u^{\eps}-\varphi)(x_{\eps})\geq(u^{\eps}-\varphi)(y)+o(\eps^2),
\end{equation}
\textbf{Case 1:} Suppose that $u^{\eps}(x_{\eps_j})>v^{\eps}(x_{\eps_j})$ for a subsequence such that $\eps_j\rightarrow 0$. Let us observe that, 
if $$J_1(u^{\eps})(z)<J_2(v^{\eps})(z)$$ 
we have that
\begin{equation}
	u^{\eps}(z)=\half J_1(u^{\eps})(z)+\half J_2(v^{\eps})(z) \quad \mbox{and} \quad 	v^{\eps}(z)=\half J_1(u^{\eps})(z)+\half J_2(v^{\eps})(z),
\end{equation}
and then we get $$u^{\eps}(z)=v^{\eps}(z)$$
in this case. 

This remark implies that when $u^{\eps}(x_{\eps_j})>v^{\eps}(x_{\eps_j})$ we have
\begin{equation}
	J_1(u^{\eps})(x_{\eps_j})\geq J_2(v^{\eps})(x_{\eps_j}).
\end{equation}
If we use the DPP \eqref{DPPEXT2} we get
\begin{equation}
	\begin{array}{ll}
		\displaystyle 0=\half(J_1(u^{\eps})(x_{\eps_j})-u^{\eps}(x_{\eps_j}))+\half \max\{J_1(u^{\eps})(x_{\eps_j})-u^{\eps}(x_{\eps_j}) , J_2(v^{\eps})(x_{\eps_j})-u^{\eps}(x_{\eps_j})\} \\[8pt]
		\displaystyle =\half(J_1(u^{\eps})(x_{\eps_j})-u^{\eps}(x_{\eps_j}))+\half J_1(u^{\eps})(x_{\eps_j})-u^{\eps}(x_{\eps_j})=J_1(u^{\eps})(x_{\eps_j})-u^{\eps}(x_{\eps_j}),
	\end{array}
\end{equation}
and using \eqref{umenosphi} we obtain
\begin{equation}
	0=J_1(u^{\eps})(x_{\eps_j})-u^{\eps}(x_{\eps_j})\leq J_1(\varphi)(x_{\eps_j})-\varphi(x_{\eps_j}),
\end{equation}
taking limit as $\eps_j\rightarrow 0$ as before we get 
\begin{equation} \label{ine-1}
	-\Delta^1_{p}\varphi(x_0)+h_1(x_0)\leq 0.
\end{equation}

We have proved before that $v$ is a subsolution to 
$$
-\Delta^1_{q}v(x)-h_2(x) = 0
$$
in the whole $\Omega$.
Therefore, as $(v-\varphi)(x_0)=0$ is a maximum of $v-\varphi$ we get 
\begin{equation} \label{ine-2}
	-\Delta^1_{q}\varphi(x_0)-h_2(x_0)\leq 0.
\end{equation}

Thus, from \eqref{ine-1} and \eqref{ine-2} we conclude that
\begin{equation}
(-\Delta^1_{p}\varphi(x_0)+h_1(x_0))+	(-\Delta^1_{q}\varphi(x_0)-h_2(x_0))\leq 0.
\end{equation}

\textbf{Case 2:} If $u^{\eps}(x_{\eps})=v^{\eps}(x_{\eps})$ for $\eps<\eps_0$.
Using the DPP \eqref{DPPEXT2} we have
\begin{equation}
	\begin{array}{ll}
		\displaystyle u^{\eps}(x_{\eps})=\half J_1(u^{\eps})(x_{\eps})+\half J_2(v^{\eps})(x_{\eps}),  \\[10pt]
		\displaystyle v^{\eps}(x_{\eps})=\half J_1(u^{\eps})(x_{\eps}) +\half J_2(v^{\eps})(x_{\eps}),  	
		\end{array}
\end{equation}
then we get
\begin{equation}
	\begin{array}{ll}
		\displaystyle \max\{J_1(u^{\eps})(x_{\eps}), J_2(v^{\eps})(x_{\eps})\}=J_2(v^{\eps})(x_{\eps}),  \\[10pt]
		\displaystyle \min\{J_1(u^{\eps})(x_{\eps}), J_2(v^{\eps})(x_{\eps})\}=J_1(u^{\eps})(x_{\eps}).
	\end{array}
\end{equation}

If we use again \eqref{umenosphi} we get
\begin{equation}
	\begin{array}{ll}
			\varphi(y)-\varphi(x_{\eps})\geq u^{\eps}(y)-u^{\eps}(x_{\eps})+o(\eps^2) \\[8pt]
			\geq v^{\eps}(y)-v^{\eps}(x_{\eps})+o(\eps^2), 
	\end{array}
\end{equation}
here we used that $u^{\eps}\geq v^{\eps}$ and $u^{\eps}(x_{\eps})= v^{\eps}(x_{\eps})$. Thus
\begin{equation}
	\begin{array}{ll}
	\displaystyle 0=\half \big(J_1(u^{\eps})(x_{\eps})-u^{\eps}(x_{\eps})\big)+\half \big(J_2(v^{\eps})(x_{\eps})-v^{\eps}(x_{\eps})\big)\\[8pt]
	\displaystyle	\leq \half \big(J_1(\varphi)(x_{\eps})-\varphi(x_{\eps})\big)+\half \big(J_2(\varphi)(x_{\eps})-\varphi(x_{\eps})\big).
	\end{array}
\end{equation}
Taking limit $\eps\rightarrow 0$ we obtain
\begin{equation}
	(-\Delta^1_{p}\varphi(x_0)+h_1(x_0))+	(-\Delta^1_{q}\varphi(x_0)-h_2(x_0))\leq 0,
\end{equation}
in the viscosity sense (taking care of the semicontinuous envelopes when the gradient of $\varphi$ vanishes). 
We have just proved that the extra condition is verified with an inequality
when we touch $u$ and $v$ from above at some point $x_0$ with a smooth test function. 

The proof that other inequality holds when we touch $u$ and $v$ from below is analogous
and hence we omit the details.

\section{Final remarks}
\label{sect-remarks}

Below we gather some brief comments on possible extensions of our results. 

\subsection{$n$ membranes} We can extend our results to the case in which we have $n$ membranes. For the PDE problem we refer to \cite{ARS,CChV,ChV}.

	We can generalize the game to an $n$-dimensional system. Let us suppose that we have for $1\leq k\leq n$
	 \begin{equation}
	 	J_k(w)(x)=\alpha_k\Big[\half \sup_{y \in B_{\eps}(x)}w(y) + \half \inf_{y \in B_{\eps}(x)}w(y)\Big]+(1-\alpha_k)\kint_{B_{\eps}(x)}w(y)dy-\eps^2h_k(x)
	 \end{equation}
 These games have associated the operators
 \begin{equation}
 	L_k(w)=-\Delta^1_{p_k}w+h_k.
 \end{equation}

Given $f_1\geq f_2\geq\dots\geq f_n$ defined outside $\Om$, we can consider the (DPP)
\begin{equation}
	\label{DPPEXTn}
	\left\lbrace
	\begin{array}{ll}
		\displaystyle u_k^{\eps}(x)=\half\max_{i\geq k}\Big\{ J_i(u_i^{\eps})\Big\}+\half \min_{l\leq k}\Big\{ J_l(u_l^{\eps})\Big\} ,
		&  x \in \Omega,  \\[10pt]
		u_k^{\eps}(x) = f_k(x), \qquad & x \in \R^{N} \backslash \Omega.
	\end{array}
	\right.
\end{equation}
for $1\leq k\leq n$.

This (DPP) is associated to a game that is played in $n$ boards. In board $k$ a fair coin is tossed and the winner is allowed to change boards but Player I can only choose to change 
to a board with index bigger or equal than $k$ while Player II may choose a board with
index smaller or equal than $k$.

The functions $(u_1^{\eps}, \cdots, u_n^{\eps})$ converge uniformly
as $\eps \to 0$ (along a subsequence) to continuous functions $\{u_k\}_{1\leq k\leq n}$ that are viscosity solutions to
the $n$ membranes problem,
	\begin{equation}
		\label{EDEXn}
		\left\lbrace
		\begin{array}{ll}
			\displaystyle u_k(x) \geq u_{k+1}(x), \qquad & \  x\in\Omega, \\[10pt]
				\displaystyle L_k(u_k)\geq 0, \quad  \quad L_{k+l}(u_{k+l})\leq 0 \qquad & \ x\in\{u_{k-1}>u_k\equiv u_{k+1}\equiv\dots\equiv u_{k+l}>u_{k+l+1}\}\cap\Om,\\[10pt] 
			\displaystyle L_k(u_k)+L_{k+l}(u_{k+l}) =0, & \  x\in\{u_{k-1}>u_k\equiv u_{k+1}\equiv\dots\equiv u_{k+l}>u_{k+l+1}\}\cap\Om,\\[10pt]
			\displaystyle L_k(u_k)=0, & \ x\in\{u_{k-1}>u_k>u_{k+1}\}\cap\Om,\\[10pt]
			u_k(x) = f_k(x), \qquad & \ x \in \partial \Omega.
		\end{array}
		\right.
	\end{equation}
for $1\leq k\leq n$. 

Notice that here the extra condition 
$$L_k(u_k)+L_{k+l}(u_{k+l}) =0, \qquad \  x\in\{u_{k-1}>u_k\equiv u_{k+1}\equiv\dots\equiv u_{k+l}>u_{k+l+1}\}\cap\Om $$
appears.

\subsection{Other operators} 
Our results can also be extended to the two membranes problem with 
different operators as soon as there are games $J_1$ and $J_2$ whose value
functions approximate the solutions to the corresponding PDEs and for which 
the key estimates of Section \ref{sect-first-game} can be proved. Namely, we need that
starting close to the boundary each player has a strategy that forces the game to end close
to the initial position in the same board with large probability and in a controlled expected number of plays regardless the choices of the other player.

For instance, our results can be extended to deal with the two membranes problem for Pucci operators (for a game related to Pucci operators we refer to \cite{BMR}). Pucci operators are uniformly elliptic and are given in terms of two positive constants, $\lambda$ and $\Lambda$
by the formulas
$$
M_{\lambda, \Lambda}^{+}(D^2 u) = \sup_{A\in L_{\lambda, \Lambda}}
tr(AD^2u) \qquad \mbox{and} \qquad 
M^{-}_{\lambda, \Lambda} (D^2 u) = \inf_{A\in L_{\lambda, \Lambda}} tr(AD^2u)
$$
with
$$
L_{\lambda, \Lambda} = \Big\{ A \in \mathbb{S}^n : \lambda Id \leq A \leq \Lambda Id
\Big\}.
$$

Notice that the extra condition that we obtain with the second game 
reads as
$$
M_{\lambda_1, \Lambda_1}^{+}(D^2 u(x)) + M_{\lambda_2, \Lambda_2}^{-}(D^2 v(x))
= h_2 (x) - h_1 (x)
$$
if we play with a game associated to the equation $M_{\lambda_1, \Lambda_1}^{+}(D^2 u(x)) + h_1(x)=0$
in the first board and with a game associated to 
$M_{\lambda_2, \Lambda_2}^{-}(D^2 u(x)) + h_1(x)=0$ in the second board.

We leave the details to the reader.

\subsection{Playing with an unfair coin modifies the extra condition}
One can also deal with the game in which the coin toss that is used to
determine if the player can make the choice to change boards or not
is not a fair coin. Assume that a coin is tossed in the first
board with probabilities $\gamma$ and $(1-\gamma)$
and in the second board with reverse probabilities, $(1-\gamma)$ and $\gamma$.
In this case the equations that are involved in the DPP read as
 \begin{equation}
	 	\label{DPPEXT.rem}
	 	\left\lbrace
	 	\begin{array}{ll}
	 		\displaystyle u^{\eps}(x)=\gamma \max\Big\{ J_1(u^{\eps})(x), J_2(v^{\eps})(x)\Big\}+ (1-\gamma) J_1(u^{\eps})(x) , \qquad
	 		&  x \in \Omega,  \\[10pt]
	 		\displaystyle v^{\eps}(x)=(1-\gamma) \min\Big\{ J_1(u^{\eps})(x), J_2(v^{\eps})(x)\Big\} + \gamma  J_2(v^{\eps})(x),
	 		&  x \in \Omega.
	 	\end{array}
	 	\right.
	 \end{equation}

	 In this case, the extra condition that we obtain is given by
	 \begin{equation} \label{extra-cond.intro.rem}
	 \gamma \big(-\Delta_{p}^{1}u(x)+ h_1(x)\big)+ (1-\gamma)\big(-\Delta_q^1 v(x)- h_2(x)\big) =0, \qquad \ x\in \Om,
	\end{equation}
	
	Notice that there are two extreme cases, $\gamma =0$ and $\gamma=1$. 
	When $\gamma =1$,  the second player can not decide to change boards
	but the first player has this possibility (with probability one) in the first board. In this case, in the limit problem the second component, $v$, is a solution to
	$-\Delta_q^1 v(x)- h_2(x)=0$ in the whole $\Omega$ and $u$ is the solution to the obstacle problem (with $v$ as obstacle from below). On the other hand, if $\gamma =0$, it is the first player who can not decide to change and the second player has the command in the second board and in this case in the limit
	it is $u$ the component that is a solution to the equation, $
	-\Delta_{p}^{1}u(x)+ h_1(x)=0$, and $v$ the one that solves the obstacle problem (with $u$ as obstacle from above). 
	
	Remark that the value functions are increasing with respect to $\gamma$, that is,
	 $u^{\eps}_{\gamma_1} (x) \leq u^{\eps}_{\gamma_2}(x)$ and 
	 $v^{\eps}_{\gamma_1} (x) \leq v^{\eps}_{\gamma_2}(x)$ for $\gamma_1 \leq \gamma_2$. Therefore, passing to the limit as $\eps \to 0$ we obtain a family
	 of solutions to the two membranes problem that is increasing with $\gamma$, 
$$u_{0} (x) \leq u_{\gamma_1} (x) \leq u_{\gamma_2}(x) \leq u_1 (x)$$ and 
	 $$v_0 (x) \leq v_{\gamma_1} (x) \leq v_{\gamma_2}(x) \leq v_1(x)$$
	  for $\gamma_1 \leq \gamma_2$.
	  
	  The pair $(u_0,v_0)$ is the minimal solution to the two membranes problem
	  in the sense that $u_0 \leq u $ and $v_0 \leq v$ for any other solution $(u,v)$.
	  In fact, since $u$ is a supersolution and $u_0$ is a solution to 
	  $
	-\Delta_{p}^{1}u(x)+ h_1(x)=0$ from the comparison principle we obtain $u_0 \leq u$.
	Then, we obtain that $v_0 \leq v$ from the fact that they are solutions to the obstacle problem from above with obstacles $u_0$ and $u$ respectively. 
	
	Analogously, the pair $(u_1,v_1)$ is the maximal solution to the two membranes problem.

\bigskip

{\bf Acknowledgements.} partially supported by 
CONICET grant PIP GI No 11220150100036CO
(Argentina), PICT-2018-03183 (Argentina) and UBACyT grant 20020160100155BA (Argentina).


\begin{thebibliography}{BH}


\bibitem{TPSS} T. Antunovic, Y. Peres, S. Sheffield and S. Somersille. {\it Tug-of-war and infinity Laplace equation with vanishing Neumann boundary condition}. Comm. Partial Differential Equations, 37(10), 2012, 1839--1869.

\bibitem{Arroyo} A. Arroyo and J. G. Llorente. {\it On the asymptotic mean value property for planar p-harmonic functions}. Proc. Amer. Math. Soc., 
144(9), (2016), 3859--3868.


\bibitem{AS} S. N. Armstrong and C. K. Smart. {\it An easy proof of Jensen's theorem on the uniqueness of infinity harmonic functions.}
Calc. Var. Partial Differential Equations, 37(3-4), (2010), 381--384.

\bibitem{ACJ} G. Aronsson, M.G. Crandall and P. Juutinen, {\it A
tour of the theory of absolutely minimizing functions}. Bull.
Amer. Math. Soc., 41, (2004), 439--505.

\bibitem{ARS} A. Azevedo, J. F. Rodrigues and L. Santos. {\it The N-membranes problem for quasilinear degenerate systems}. Interfaces Free Bound., 7(3), (2005), 319--337.

\bibitem{BChMR} P. Blanc, F. Charro, J. D. Rossi and J. J. Manfredi. \emph{A nonlinear Mean Value Property for the Monge-Amp\`ere operator}. 
J. Convex Analysis JOCA. 28(2), (2021), 353--386.

\bibitem{BMR} P. Blanc, J. J. Manfredi and J. D. Rossi. {\it Games for Pucci’s maximal operators}. Jour. Dyn. Games, 6(4), (2019), 277--289.

\bibitem{BlancPR} P. Blanc, J. P. Pinasco and J. D. Rossi. {\it Maximal operators for the $p-$Laplacian family}. 
Pacific J. Math., 287(2), (2017), 257--295.


\bibitem{BR} P. Blanc and J. D. Rossi. {\it Games for eigenvalues of the Hessian and concave/convex envelopes.} 
J. Math. Pures Appl., 127, (2019), 192--215.

\bibitem{BRLibro} P. Blanc and J. D. Rossi. {Game Theory and Partial Differential Equations.}
De Gruyter Series in Nonlinear Analysis and Applications. Vol. 31. 2019.
ISBN 978-3-11-061925-6.
ISBN 978-3-11-062179-2 (eBook).

\bibitem{Caffa1} 
L. Caffarelli, D. De Silva and O. Savin. {\it The two membranes problem for different operators.} Ann. l'Institut Henri Poincare C, Anal. non lineaire,
 34(4), (2017), 899--932.
 
 \bibitem{Caffa2} L. Caffarelli, L. Duque and H. Vivas. {\it The two membranes problem for fully nonlinear operators.} Discr. Cont. Dyn. Syst, 38(12), (2018), 6015--6027.


\bibitem{CChV} S. Carillo, M. Chipot and G. Vergara-Caffarelli. {\it The N-membrane problem with nonlocal
constraints}, J. Math. Anal. Appl., 308(1), (2005), 129--139.

\bibitem{ChGAR} F. Charro, J. Garcia Azorero and J. D. Rossi. {\it A mixed problem for
the infinity laplacian via Tug-of-War games.} Calc. Var. Partial
Differential Equations, 34(3), (2009), 307--320.

\bibitem{ChV} M. Chipot and G. Vergara-Caffarelli. {\it The N-membranes problem}, Appl. Math. Optim., 13(3),
(1985), 231--249.

\bibitem{Cran} M.G. Crandall. {\it A visit with the $\infty$-Laplace equation}, Calculus of variations and nonlinear partial differential equations, 
Lecture Notes in Math., Vol. 1927, Springer, Berlin, 2008, 75--122.


\bibitem{CIL} M.G. Crandall,  H. Ishii and  P.L. Lions. {\it User's guide to viscosity solutions of second order partial differential equations}. Bull. Amer. Math. Soc., 27, (1992), 1--67.



\bibitem{Doob} J.L. Doob, {\it What is a martingale ?}, Amer. Math. Monthly, 78(5), (1971), 451--463.

\bibitem{Doob2} J.L. Doob, Classical Potential Theory and Its Probabilistic Counterpart. Classics in Mathematics. Springer. 2001.

\bibitem{Doob3} J.L. Doob, {\it Semimartingales and subharmonic functions}. Trans. Amer. Math.
Soc., 77, (1954), 86--121. 

\bibitem{Hunt} G. A. Hunt. {\it Markoff processes and potentials I, II, III}, Illinois J. Math. 1
(1957), 44--93, 316--369; ibid. 2 (1958), 151--213. 


\bibitem{I} M. Ishiwata, R. Magnanini and H. Wadade. {\it A natural approach to the asymptotic mean value property for the p-Laplacian}. Calc. Var. Partial Differential Equations, 56(4), (2017), Art. 97, 22 pp.



\bibitem{Kac} M. Kac. {\it Random Walk and the Theory of Brownian Motion}.
Amer. Math. Monthly, 54(7), (1947), 
369--391.

\bibitem{Kaku} S. Kakutani, {\it Two-dimensional Brownian motion and harmonic functions}. Proc.
Imp. Acad. Tokyo, 20, (1944), 706--714. 

\bibitem{KMP} B. Kawohl, J.J. Manfredi and M. Parviainen. {\it Solutions of nonlinear PDEs in the sense of averages}. J. Math. Pures Appl., 97(3), (2012), 173--188.

\bibitem{FFF} A. W. Knapp. {\it Connection between Brownian Motion and
Potential Theory}. Jour. Math. Anal. Appl., 12, (1965), 328--349.  


\bibitem{Lewicka} M. Lewicka. A Course on Tug-of-War Games with Random Noise.
Introduction and Basic Constructions. Universitext book series. Springer, (2020).

\bibitem{ML} M. Lewicka and J. J. Manfredi. {\it The obstacle problem for the $p-$laplacian via optimal stopping of tug-of-war games.}
Prob. Theory Rel. Fields, 167(1-2), (2017), 349--378.



\bibitem{LM} P. Lindqvist and J. J. Manfredi. {\it On the mean value property for the $p-$Laplace equation in the plane}. Proc. Amer. Math. Soc.,
144(1), (2016), 143--149.


\bibitem{QS}
Q. Liu and A. Schikorra.
\textit{General existence of solutions to dynamic programming principle.}
 Commun. Pure Appl. Anal., 14(1), (2015), 167--184.
 
 

\bibitem{LPS}
H. Luiro, M. Parviainen, and E. Saksman.
{\it Harnack's inequality for p-harmonic functions via stochastic games.}
Comm. Partial Differential Equations, 38(11), (2013), 1985--2003.



\bibitem{MPR} J. J. Manfredi, M. Parviainen and J. D. Rossi. {\it An asymptotic mean value characterization for $p-$harmonic functions}. Proc. Amer. Math. Soc., 138(3), (2010), 881--889. 

 \bibitem{MPRa} J. J. Manfredi, M. Parviainen and J. D. Rossi.
\textit{Dynamic programming principle for tug-of-war games with noise.}
ESAIM, Control, Opt. Calc. Var., 18, (2012), 81--90.

\bibitem{MPRb} J. J. Manfredi, M. Parviainen and J. D. Rossi.
\textit{On the definition and properties of p-harmonious functions.}
Ann. Scuola Nor. Sup. Pisa, 11, (2012), 215--241.

\bibitem{MPRparab} J. J. Manfredi, M. Parviainen and J.~D.~Rossi.
{\it An asymptotic mean value  characterization for  a class of
nonlinear
 parabolic equations related to tug-of-war games.}
SIAM J. Math. Anal., 42(5), (2010), 2058--2081.


\bibitem{MRS} J. J. Manfredi, J. D. Rossi and S. J. Somersille.
{\it An obstacle problem for tug-of-war games.} 
Com. Pure Appl. Anal., 14(1), (2015), 217--228. 



\bibitem{nosotros} A. Miranda and J. D. Rossi. {\it A game theoretical approach for a nonlinear system driven by elliptic operators.} 
SN Partial Diff. Eq. Appl., 1(4), art. 14, pp 41, (2020).

\bibitem{nosotros2} A. Miranda and J. D. Rossi. {\it A game theoretical approximation for a parabolic/elliptic system with different operators.} 
Disc. Cont. Dyn. Syst., 43, (2023), 1625–1656.


\bibitem{Mitake} H. Mitake and H. V. Tran, {\it 
Weakly coupled systems of the infinity Laplace equations}, Trans. Amer. Math. Soc., 369 (2017), 1773--1795.


\bibitem{PSSW} Y. Peres, O. Schramm, S. Sheffield and D. Wilson,
{\it Tug-of-war and the infinity Laplacian.} J. Amer. Math. Soc.,
22, (2009), 167--210.


\bibitem{PS} Y. Peres and S. Sheffield, {\it Tug-of-war with noise:
a game theoretic view of the $p$-Laplacian}, Duke Math. J., 145(1), (2008), 91--120.

\bibitem{R} J. D. Rossi. {\it Tug-of-war games and PDEs.} Proc.
Royal Soc. Edim. 141A, (2011), 319--369.

\bibitem{S} L. Silvestre. {\it The two membranes problem}, Comm. Partial Differential Equations, 30(1-3), (2005), 245--257.

\bibitem{VC} G. Vergara-Caffarelli, {\it Regolarita di un problema di disequazioni variazionali relativo a due
membrane}, Atti Accad. Naz. Lincei Rend. Cl. Sci. Fis. Mat. Natur. (8) 50, (1971), 659–-662
(Italian, with English summary).

\bibitem{Vivas} H. Vivas. The two membranes problem for fully nonlinear local and nonlocal operators. PhD Thesis dissertation. UT Austin. (2018). https://repositories.lib.utexas.edu/handle/2152/74361

\bibitem{Williams} D. Williams, Probability with martingales, Cambridge University Press, Cambridge,
1991.


\end{thebibliography}
\end{document}